\documentclass[12pt, nohyperref]{article}
\usepackage{fullpage}
\usepackage{times}
\usepackage[total={6.5in,8.75in},
top=1.2in, left=0.9in, includefoot]{geometry}

\newif\ificml

\icmltrue 
\icmlfalse 

\usepackage{microtype}
\usepackage{graphicx}
\usepackage{subfigure}
\usepackage[makeroom]{cancel}
\usepackage{booktabs} 
\usepackage{makecell}
\usepackage[colorlinks=true,citecolor=blue]{hyperref}%
\usepackage{multirow}
\usepackage[table, svgnames, dvipsnames]{xcolor}
\usepackage{makecell, cellspace, caption}
\usepackage{booktabs,caption}
\usepackage[flushleft]{threeparttable}

\newcommand{\nosemic}{\renewcommand{\@endalgocfline}{\relax}}
\newcommand{\dosemic}{\renewcommand{\@endalgocfline}{\algocf@endline}}
\let\oldnl\nl
\newcommand{\nonl}{\renewcommand{\nl}{\let\nl\oldnl}}

\usepackage{amsmath, nccmath}
\usepackage{amssymb}
\usepackage{mathtools}
\usepackage{enumitem}
\usepackage{amsthm}
\usepackage{lipsum}

\usepackage[capitalize,noabbrev]{cleveref}
\allowdisplaybreaks

\usepackage[utf8]{inputenc} 
\usepackage[T1]{fontenc}    
\usepackage{hyperref}       
\usepackage{url}       
\usepackage{booktabs}       
\usepackage{amsfonts}      
\usepackage{nicefrac}      
\usepackage{microtype}
\usepackage{cite}
\usepackage{amsmath,amssymb,amsfonts,amsthm}
\usepackage{algorithmic}
\usepackage{graphicx}
\usepackage{textcomp}
\usepackage{color}
\usepackage{enumitem} 
\usepackage[ruled,vlined,linesnumbered,inoutnumbered]{algorithm2e}
\usepackage{pifont}
\usepackage{bm}
\usepackage{caption}

\newtheorem{theorem}{Theorem}
\newtheorem{corollary}{Corollary}
\newtheorem{lemma}{Lemma}

\newtheorem{assumption}{Assumption}

\theoremstyle{remark}
\newtheorem{remark}{Remark}

\def\o{\omega}
\def\R{\mathbb{R}}
\def\E{\mathbb{E}}
\def\G{\mathcal{G}}
\def\H{\mathcal{H}}

\def\cM{\mathcal{M}}
\def\P{\mathbb{P}}

\def\M{\mathbf{M}}
\def\O{\mathcal{O}}
\def\I{\mathbf{I}}
\def\K{\mathbf{K}}
\def\W{\mathbf{W}}

\def\bH{\mathbf{H}}

\newcommand{\p}{\mathbf{p}}
\newcommand{\x}{\mathbf{x}}
\newcommand{\w}{\mathbf{w}}

\newcommand{\bv}{\mathbf{v}}

\newcommand{\bxi}{\boldsymbol{\xi}}
\newcommand{\bth}{\boldsymbol{\theta}}

\newcommand{\one}{\mathbf{1}}
\newcommand{\ox}{\bar{\mathbf{x}}}

\newcommand{\og}{\overline{\gamma}}
\newcommand{\y}{\mathbf{y}}

\newcommand{\m}{\mathbf{m}}

\newcommand{\N}{\mathcal{N}}
\newcommand{\A}{\mathcal{A}}

\newcommand{\F}{\mathcal{F}}
\newcommand{\inpd}[1]{\langle #1 \rangle}

\newcommand{\argmin}{\mathop{\mathrm{arg\,min}}}

\newcommand{\norm}[1]{\|#1\|}

\newcommand{\abs}[1]{\left| #1 \right|}

\SetKwInput{KwInit}{Initialization}
\SetKwInput{KwOut}{Output}
\allowdisplaybreaks

\newcommand{\cvrg}{\textsc{CLIP-VRG}}

\title{
        Secure Distributed Optimization Under Gradient Attacks
}

\author{Shuhua Yu,
        Soummya Kar\thanks{This work was supported in part by the Office of Naval Research under Grant \# N00014-21-1-2547.}\\
        \\
        Department of Electrical and Computer Engineering,\\
        Carnegie Mellon University, Pittsburgh, PA\\
        \{shuhuay, soummyak\}@andrew.cmu.edu}

\date{October 27, 2022}

\begin{document}
        
        \maketitle

        \vskip 0.3in
        
        
        
        
        
        \begin{abstract}
In this paper, we study secure distributed optimization against arbitrary gradient attack in multi-agent networks. In distributed optimization, there is no central server to coordinate local updates, and each agent can only communicate with its neighbors on a predefined network. We consider the scenario where out of $n$ networked agents, a fixed but unknown fraction $\rho$ of the agents are under arbitrary gradient attack in that their stochastic gradient oracles return arbitrary information to derail the optimization process, and the goal is to minimize the sum of local objective functions on unattacked agents. We propose a distributed stochastic gradient method that combines local variance reduction and clipping (\cvrg). We show that, in a connected network, when unattacked local objective functions are convex and smooth, share a common minimizer, and their sum is strongly convex, \cvrg\ leads to almost sure convergence of the iterates to the exact sum cost minimizer at all agents. We quantify a tight upper bound of the fraction $\rho$ of attacked agents in terms of problem parameters such as the condition number of the associated sum cost that guarantee exact convergence of \cvrg, and characterize its asymptotic convergence rate. Finally, we empirically demonstrate the effectiveness of the proposed method under gradient attacks in both synthetic dataset and image classification datasets.
\end{abstract}

        \section{Introduction}
In this paper, we study the problem of secure distributed optimization in peer-to-peer multi-agent networks under arbitrary gradient attacks. In distributed optimization over $n$ networked agents, each agent $i \in [n]$ holds a local objective function $f_i$, has access to stochastic gradients of its local $f_{i}$ via a local and private stochastic gradient oracle, and may only communicate with its direct neighbors defined by an inter-agent communication graph to cooperatively minimize the aggregated objective function $\sum_{i\in[n]} f_i$. The sum-cost minimization and its stochastic variants as described above have emerged as natural abstractions of performing various distributed signal processing and machine learning tasks and seen extensive research over the past decade with the primary focus of building distributed stochastic gradient like procedures based on consensus \cite{nedic2009distributed, kar2012distributed} or diffusion processes \cite{chen2012diffusion} with proven convergence or learning guarantees \cite{nedic2020distributed}. 

This paper studies the adversarial setting where a fixed but unknown $\rho$-fraction agents are under gradient attack in that their stochastic gradient oracles return arbitrary adversarial information when queried during algorithm execution. In distributed network based settings such a scenario arises in which the local cost functions (often reflecting the local data at the agents) are manipulated by an adversary, corresponding to the practical class of data injection attacks.\footnote{In what follows, we will refer to such data injection attacks as gradient attacks as the constructed optimization procedures are based on first order gradient optimization. Although for abstraction purposes we model the attacks as affecting the gradients, it follows from a straightforward inspection of our algorithms that the proposed constructions carry over to other types of attacks that directly manipulate the data or cost functions at the agents.} Denoting by $\mathcal{A}$ the set of agents, unknown apriori, whose gradients are potentially attacked and by $\mathcal{N}$ the non-attacked agents such that $|\mathcal{A}|+|\mathcal{N}|=n$, in the adversarial scenario, instead of minimizing the global aggregate $\sum_{i\in[n]} f_i$, we aim to solve
\begin{align}
\label{eq:dist-mini}
  \text{minimize}_{\x \in \R^d} f(\x) := \frac{1}{\abs{\N}} \sum_{i \in \N} f_i(\x). 
\end{align}
When $\N = [n]$, problem \eqref{eq:dist-mini} reduces to the classical distributed optimization formulation (in non-adversarial environments) as discussed above.

Many machine learning and statistical inference tasks are currently being implemented via decentralized computation paradigms such as Federated Learning \cite{mcmahan2017communication, li2020federated} instead of the classical single server paradigm, to deal with the scalability, robustness and privacy issues. As server-worker paradigms, where there is a central server to coordinate local model updates, are prone to single server point of failures and bottlenecks, \textit{fully} distributed paradigms that distribute the computation task among multiple entities have gained increasing attention \cite{xin2020general, nedic2020distributed, lian2017can}. However, distributed data and communication bring up data integrity concerns. As many distributed optimization methods are based on stochastic gradient computation, gradient attacks induced by malicious data injection is a serious security concern that needs to be addressed. For examples, the adversary can inject malicious data points to some participating agents in decentralized machine learning training \cite{yu2021resilient}, or corrupt sensor measurements in statistical inference over sensor networks \cite{chen2018resilient, chen2018internet}. In these settings, undefended distributed algorithms can be arbitrarily misled by gradient attack, and thus motivates the formulation \eqref{eq:dist-mini}. 

Our work is closely related to Byzantine-robust distributed optimization. Byzantine attack \cite{lamport1982byzantine} is the most difficult threat model considered in the distributed optimization literature in that Byzantine agents can deviate from the prescribed algorithm and send arbitrarily adversarial messages to its neighbors. In contrast, the threat model considered in this paper assumes that the attacked agents still preserve normal computation and communication capabilities, i.e. still follow the prescribed algorithmic procedures. Our attack model is weaker but also more practical since one may employ cryptographic protocols to test whether agents are Byzantine while it is harder to check the legitimacy of distributed and even heterogeneous data. To couple with Byzantine agents, different approaches to aggregate information from neighbors \cite{wu2022byzantine} have  been proposed. The works \cite{su2020byzantine, sundaram2018distributed, yang2019byrdie, fang2022bridge} use trimmed average of model updates from neighbors, but this approach requires that the majority of non-Byzantine agent's neighbors are also non-Byzantine, while in our threat model we do not have such strict requirements on the distribution of attacked agents over networks. If data are identically and independently distributed (i.i.d.), \cite{guo2021byzantine} proposes to use local data points to evaluate the models communicated from neighbors and only trust those with good performance on local data. Algorithmically, the most related method to \cvrg\ is \textsc{SSCLIP} proposed in \cite{he2022byzantine} that combines local momentum and self-centered clipping on differences between local and incoming models. In contrast, \cvrg\ uses decaying stepsizes to achieve local variance reduction, and the clipping operator in \cvrg\ is applied on gradient estimator themselves instead of differences. More practically, \cvrg\ uses predefined clipping sequences, whereas, \textsc{SSCLIP} relies on information that may not be accessible in a distributed environment for obtaining their clipping thresholds. In addition, 
\cite{peng2021byzantine} proposes a TV-regularized approximation of the Byzantine-free optimization problem along with a subgradient method that converges to a neighborhood of the optimal solution. Exact minimum is also achievable if some redundancy condition and a complete communication graph are assumed \cite{gupta2020fault}. Our ``common minimizer'' assumption is mildly stronger than the redundancy assumption in \cite{gupta2020fault}, but our algorithm applies to general connected network topology and is based on stochastic gradient while \cite{gupta2020fault} requires full gradient. 

Our work is also related to resilient distributed parameter estimation in terms of threat model. The papers \cite{chen2018resilient, chen2019resilient} consider sensor attack that can arbitrarily manipulate sensor measurements but the computation and communication capabilities of sensors remain intact, which is similar to the gradient attack studied in this work. To counter the sensor measurement attack, a distributed estimator based on saturated local update is proposed in \cite{chen2018resilient, chen2019resilient} which is the first clipping based method to achieve resilience in distributed inference to the best of our knowledge. From an optimization perspective, this method is indeed a gossip-type distributed stochastic gradient method with local gradient clipping. In contrast, \cvrg\ applies to a wider range of distributed optimization models comparing to the de facto least-squares model considered in \cite{chen2018internet, chen2019resilient}. From the algorithmic perspective, \cvrg\ uses constant mixing matrix while \cite{chen2018internet, chen2019resilient} employ a sequence of approximation type mixing matrices following from the \textit{consensus+innovations} framework \cite{kar2013consensus+}, and \cvrg\ further involves local variance reduction for gradient estimation. We refer the reader to \cite{chen2018internet} and references therein for a broader survey on more countermeasures to achieve resilient distributed inference in sensor networks.

\textit{Other Related Works}. Distributed optimization has been extensively studied leading to various distributed algorithms \cite{tsitsiklis1986distributed}, including distributed (sub)gradient descent methods \cite{nedic2009distributed, ram2010distributed}, gradient tracking based methods \cite{qu2017harnessing, xin2020general, pu2021distributed}, acceleration \cite{jakovetic2014fast}, variance reduction \cite{shi2015extra}, primal-dual methods \cite{duchi2011dual, chang2014distributed, wang2021distributed}, etc. Distributed optimization has also been studied taking into account different communication topologies \cite{wang2019matcha,ying2021exponential}, compressed communication \cite{koloskova2019decentralized}, data heterogeneity \cite{vogels2021relaysum} and data privacy \cite{cyffers2022muffliato}. Although adversarial robustness of distributed optimization is relatively less studied, in server-worker type setups with a central trustworthy server, several approaches to achieve Byzantine robustness have been proposed. In these approaches the central server employs robust gradient aggregators such as the median \cite{chen2017distributed, blanchard2017machine, guerraoui2018hidden, xie2018zeno}, geometric median \cite{pillutla2022robust}, concentration filtering \cite{alistarh2018byzantine, data2021byzantine}, signSGD \cite{bernstein2018signsgd, li2019rsa}, gradient clipping \cite{karimireddy2021learning}, and worker momentum \cite{el2021distributed}. When the central server also has access to the training data, the server can score the incoming gradients and abandon those abnormal ones \cite{xie2020zeno++, regatti2020bygars}, and may also reach exact minimum by exploiting redundancy \cite{chen2018draco}. In the case that the probability of an agent being Byzantine or trustworthy follows a two-state Markov Chain, \cite{turan2022robust} proposes a method with temporal and spatial robust aggregation, and gradient normalization. In addition, in the decentralized optimization setting with a trustworthy server, Byzantine robustness combined with other challenging constraints have also been studied, including privacy \cite{zhu2022byzantine}, asynchronous decentralized computing \cite{damaskinos2018asynchronous}, and in particular data heterogeneity. Given distributed heterogeneous data, attackers may take advantage of the variance of good workers over time \cite{baruch2019little} making it hard to distinguish Byzantine workers, methods such as bucketing \cite{karimireddy2020byzantine}, RSA \cite{li2019rsa} and concentration filtering \cite{data2021byzantine} have been developed to counteract this issue. 

\textit{Main Contributions}. The main contributions of this paper are as follows: (1) We consider an arbitrary gradient attack model that is relatively unexplored in the context of distributed optimization; (2) We develop a distributed stochastic gradient method, i.e.,  \cvrg, that combines local variance reduced gradient estimation and clipping, and analytically and empirically illustrate its robust performance against gradient attacks; (3) For convex and smooth unattacked local objective functions that share a common minimizer, if the sum of unattacked objective functions is strongly convex, we prove that \cvrg\ asymptotically and almost surely converges to the exact minimizer as long as proportion $\rho$ of attacked agents satisfies $\rho < 1/(1 + \kappa)$, where $\kappa$ is the condition number of the aggregated unattacked objective function.

\textit{Notations}. We use $[n] = \{1, 2, \ldots, n\}$ to denote the set of all network agents, and $| \cdot |$ to denote cardinality for an argument set such as $\N$. We use $\norm{\cdot}$ to denote Euclidean norm for vectors and $\norm{\cdot}_2$ for spectral norm of matrices, respectively. We use $\text{diag}(\cdot)$ to denote the diagonal matrix whose diagonal entries are components of the argument vector. We use $\one_p$ to denote the column vector of ones of length $p$, and bold lower case and upper case letters to denote vectors and matrices, respectively. Equality or inequality that involves random variables holds true almost surely. Random variables with superscript or subscript $\omega$ correspond to the sample path $\omega$ in sample space $\Omega$.

\textit{Organizations}. In Section \ref{sec:problem}, we formalize the problem assumptions and discuss their implications. In Section \ref{sec:alg-main}, we develop \cvrg\ and present our main theoretical results. Section \ref{sec:experiments} details the implementations and empirical performance of \cvrg\ for a regularized logistic regression model on both synthetic dataset and image classification datasets. The proofs of the main results are provided in Section \ref{sec:proofs}, whereas, Section \ref{sec:conclusion} concludes the paper.

        \section{Problem setup} 
\label{sec:problem}
Referring to the scenario in (1), recall, by $\A$ we denote the set of agents whose stochastic gradient oracles are arbitrarily manipulated by some adversary. For simplicity, let $a = \abs{\N}, b = \abs{\A}, a + b = n$, and the fraction of attacked agents $\rho = b/n$. Agents aim to minimize $f$ as defined in \eqref{eq:dist-mini} by local computations and communications with neighbors, i.e., in a distributed manner, as is common in the distributed consensus or gossip based computing literature~\cite{nedic2009distributed, kar2012distributed, chen2012diffusion}. Suppose the stochastic gradient oracle on agent $i$ returns $\m_i(\x)$ at query $\x$. We make the following assumption on the stochastic gradient oracles that also formalizes the threat model considered in this work.

\begin{assumption}
\label{as:noise}
Each agent $i \in [n]$ has access to a stochastic gradient oracle that returns $\m_i^t$ at query $\x_i^t$. For $i \in \N, \m_i^t = \nabla f_i(\x_i^t) + \bxi_i^t$ with $\{\bxi_i^t\}_{i \in [n], t \ge 0}$ being mutually independent, $\E(\bxi_i^t) = 0$, and $ \E(\|\bxi_i^t\|^2) \le \sigma^2, \forall i \in [n], t \ge 0$. For $i \in \A$, $\m_i^t$ is arbitrary. The sets $\N$ and $\A$ are fixed but apriori unknown. We further assume that agents in $\mathcal{A}$, although suffering from potential gradient (data) attacks, are otherwise non-adversarial and follow recommended algorithmic protocols as specified. 
\end{assumption}

\begin{remark}
Our gradient model works for general expected risk minimization. In machine learning or statistical inference tasks, $f_i$ can be defined as
\begin{align*}
    f_i(\x) := \E_{\vartheta \sim \mathcal{D}_i} F_i(\x, \vartheta), 
\end{align*}
where $F_i$ is defined on local data samples $\vartheta_i$ that with distribution $\mathcal{D}_i$, and the unattacked stochastic gradient oracles return $\nabla F_i(\x, \vartheta)$. If $f_i$ is defined as a finite sum over data points, then one can still sample stochastic gradients from mini-batches. 

Note that this threat model only involves attack on gradient oracles but attacked agents will still follow the recommended algorithmic procedures, i.e., we assume data injection attacks on a subset of networked agents but otherwise the agents themselves are non-Byzantine. The noise model for regular agents is standard. The gradient attacks are allowed to be arbitrary, which subsumes all specific data injection attack designs.
\end{remark}



\begin{assumption}
\label{as:smooth}
For each unattacked agent $i \in \N$, $f_i$ is $L$-smooth, i.e., $\forall \x, \y \in \R^d$, we have
\begin{align*}
    \|\nabla f_i(\x) - \nabla f_i(\y)\| \le L\|\x - \y\|.
\end{align*}
\end{assumption}

\begin{assumption}
\label{as:cvx} 
For each unattacked agent $i \in \N$, $f_i$ is convex and twice differentiable. The average objective on unattacked agents
$f$ is $\mu$-strongly convex, i.e., $\forall \x, \y \in \R^d$,
\begin{align*}
    f(\y) \ge f(\x) + \inpd{\nabla f(\x), \y - \x} + \frac{\mu}{2}\|\y - \x\|^2.  
\end{align*}
Since $f$ is also $L$-smooth by definition \eqref{eq:dist-mini} and Assumption \ref{as:smooth}, we can define the condition number of $f$ as $\kappa := L/\mu$.

\end{assumption}
Note that this strong convexity assumption is made on the average of local objective functions in $\N$ instead of for every single $i \in \N$.

\begin{assumption}
\label{as:zero-gd}
We assume that there exists a common global minimizer $\x^* = \argmin_{\x \in \R^d} f_i(\x)$ for every $i \in \N$.
\end{assumption}

\begin{remark}
Note that we do not assume that $\x^*$ in Assumption \ref{as:zero-gd} is the unique minima of each $f_{i}$ in $\N$. In particular, the individual $f_{i}$'s may have multiple non-overlapping minima that are not minimizers of the global function $f$, and hence the agents need to collaborate to find the minimizer of $f$. Clearly, in the adversarial setting, this task is further complicated by the presence of an adversary that aims to hinder such coordination to converge to a minimizer of the global cost $f$.

This assumption clearly subsumes the case where each agent $i \in \N$ performs the same machine learning tasks on i.i.d. (independently and identically distributed) data distributions by optimizing the same strongly convex such as logistic regression. On the other hand, it also includes the case where data is not i.i.d across agents such as in distributed sensing. For example, in the problem of distributed linear estimation where the environment parameter is globally observable \cite{chen2019resilient}, there exists a unique global minimizer, but the data distribution on different agents can be heterogeneous when agents have different observation matrices, and hence each agent can have minimizer that is not globally optimal. We also elaborate this case by experiments in Section \ref{sec:experiments} part A.
\end{remark}

Distributed network-based algorithms typically involve one information mixing step to aggregate decision variables from neighboring agents, the \emph{neighborhoods} being specified by an inter-agent communication graph $\G$.
\begin{assumption}
\label{as:connected}
The inter-agent communication graph $\G$ is undirected and connected. 
\end{assumption}

Suppose in each step of local computation each agent holds variable $\x_i^{t+1/2}$, we consider fixed nonnegative mixing parameters $w_{ij}$ in update $\x_i^t = \sum_{j =1}^n w_{ij} \x_j^{t+1/2}$. We make the following standard assumption on the mixing matrix $\W$ composed of entries $w_{ij}$.
\begin{assumption} 
\label{as:mixmatrix}
The nonnegative weight matrix $\W$ satisfies that $w_{ij} \neq 0$ only if there is a communication link between agent $i$ and $j$ in graph $\G$, or $i=j$. Further, $\W$ is real symmetric, doubly stochastic, and has eigenvalues $1 = \lambda_1(\W) > \abs{\lambda_2(\W)} \ge \ldots \ge \abs{\lambda_n(\W)}$ with $\beta := \abs{\lambda_2(\W)} \in [0, 1]$.
\end{assumption}

Note that, under Assumption \ref{as:connected}, there exists a $\W$ satisfying Assumption \ref{as:mixmatrix} (see~\cite{dimakis2010gossip}). In particular, in the special case when the graph $\G$ is complete, we may choose $\W = (1/n) \one \one^\top$ which recovers computing in centralized scenarios.

\begin{assumption}
\label{as:resilience-ratio}
The fraction $\rho$ of attacked agents satisfies $\rho < 1/(1 + \kappa)$. 
\end{assumption}
\begin{remark}
We use a simple example to show that $\rho < 1/(1 + \kappa)$ is actually tight for the considered attack model, i.e., recovering the exact minimum is impossible when $\rho \ge 1/(1 + \kappa)$. Suppose a set of $2m$ agents hold the same object function $x^2$ (note $\kappa=1$ in this case), and there exists an algorithm $\cM$ with which each agent can resiliently find the optimal solution $0$ when $m$ agents, i.e., $\rho = 1/(1+\kappa) = 1/2$, are under arbitrary gradient attack. Let every attacked agent simulate objective function $(x - 1)^2$, then 
\begin{align}
\label{eq:attackhypothesis}
    \cM\Big( \underbrace{x^2, \ldots, x^2}_{m \text{  regular agents}}, \underbrace{(x-1)^2, \ldots, (x-1)^2}_{m \text{ attacked agents}} \Big) = 0. 
\end{align}
Now, if we set the local objectives on all $2m$ agents as $(x-1)^2$, and $m$ agents are under arbitrary gradient attack and say they simulate  local objective functions of $x^2$, then, by our hypothesis \eqref{eq:attackhypothesis}, and Assumption \ref{as:noise} that the set of attacked agents are unknown to $\cM$, $\cM$ will lead to solution 0 instead of the true optimizer $1$. Hence, such $\cM$ will not exist and the upper bound in Assumption \ref{as:resilience-ratio} cannot be relaxed.

\end{remark}


        \section{Algorithm development and main results}
\label{sec:alg-main}
We next develop our algorithm \cvrg, see the tabular description Algorithm \ref{alg:rdgd} for details. In distributed computation setup, each agent $i \in [n]$ holds a local decision variable $\x_i^t \in \R^d$ at iteration $t$. Recall that each regular agent $i \in \N$ computes a stochastic gradient $\m_i^t $. By the smoothness assumption in \ref{as:smooth}, if the distance between consecutive iterates converges to $0$, the difference between consecutive true gradients also converges to $0$. Thus, we employ a local recursive averaging scheme to reduce the variance of local gradient estimations on regular agents. To this end, we develop a recursive gradient estimator $\bv_i^t$ computed as
\begin{align}
\label{eq:eta_vr}
    \bv_i^0 = \m_i^0, \ 
    \bv_i^{t+1} = (1 - \eta_t)\bv_i^{t} + \eta_t  \m_i^{t+1}, \forall t \ge 0, 
\end{align}
where
\begin{align}
\label{eq:eta} 
    \eta_t = c_\eta ( t+ \varphi)^{-\tau_\eta} \in (0, 1),
\end{align}
for some positive constants $c_\eta, \tau_\eta$ and positive integer $\varphi$ to be specified.

We use simple clipping to combat arbitrary gradient attack from the adversary. Specifically, we use the following distributed clipped gradient method with decaying clipping threshold, i.e., for each $i \in [n]$,
\begin{align}
\label{eq:local-rdgd}
    \x_i^{t+1} = \sum_{j = 1}^n w_{i j} \big(\x_j^t - \alpha_t k_j^t \bv_j^t \big),
\end{align}
where $w_{ij}$ are entries of matrix $\W$ as in Assumption \ref{as:mixmatrix}, the clipping coefficient $k_i^t$ is defined as
\begin{align*}
     k_i^t = 
     \begin{cases}
     1, & \norm{\bv_i^t} \le \gamma_t, \\
     \gamma_t \norm{\bv_i^t}^{-1}, & \norm{\bv_i^t} > \gamma_t,
     \end{cases}
\end{align*}
and the clipping threshold $\gamma_t$ and stepsizes $\alpha_t$ are defined as
\begin{align}
    & \gamma_t = c_\gamma (t+\varphi)^{-\tau_\gamma} \in (0, 1), \label{eq:gamma} \\
    & \alpha_t = c_\alpha (t + \varphi)^{-\tau_\alpha}\in (0, 1), \label{eq:alpha}
\end{align}
for some positive constants $c_\gamma, \tau_\gamma, c_\alpha, \tau_\alpha$ to be chosen. We outline the procedures of \cvrg\ in Algorithm \ref{alg:rdgd}.

\begin{algorithm}
\SetAlgoLined
\KwIn{$\eta_t, \alpha_t, \gamma_t$.}
\KwInit{$\x_i^0 = \x_j^0, \forall i, j \in [n]$.}
 \For{$t = 0, \ldots, T-1$}{
  \For{agent $i \in [n]$ in parallel}{
  Query stochastic gradient oracle that returns $\m_i^t$\;
  Update $\bv_i^{t} = (1 - \eta_t)\bv_i^{t-1} + \eta_t  \m_i^{t}$\;
  Compute $k_i^t = 
     \begin{cases}
     1, & \norm{\bv_i^t} \le \gamma_t, \\
     \gamma_t \norm{\bv_i^t}^{-1}, & \norm{\bv_i^t} > \gamma_t,
     \end{cases}$\; 
  Send $\x_i^t - \alpha_t k_i^t \bv_i^t$ to all neighbors of agent $i$\; 
  Update $\x_i^{t+1} =  \sum_{j=1}^n w_{ij} \big(\x_j^t - \alpha_t k_j^t \bv_j^t \big)$\;
 }
 }
 \KwOut{$\{\x_i^{T}\}_{i \in [n]}$.}
 \caption{\cvrg} 
 \label{alg:rdgd} 
\end{algorithm}

\begin{theorem}
\label{thm:rdgd}
Under Assumptions \ref{as:noise}-\ref{as:mixmatrix}, suppose that $\alpha_t, \gamma_t, \eta_t$ are taken as in \eqref{eq:eta}\eqref{eq:gamma}\eqref{eq:alpha} with $\tau_\eta = 2(\tau_\alpha + \tau_\gamma)/3, 2\tau_\gamma < \tau_\alpha < \min(1, 1 - \tau_\gamma)$. Then, for all $i \in [n]$, for every $0 < \tau < \min(\tau_\gamma, (\tau_\alpha - 2\tau_\gamma)/3)$, we have
\begin{align*}
    \P\left( \lim_{t \rightarrow \infty} (t+1)^{\tau} \norm{\x_i^t - \x^*} = 0 \right) = 1.
\end{align*}
\end{theorem}

\begin{corollary}
\label{cor:funcsub}
Under Assumptions \ref{as:noise}-\ref{as:mixmatrix}, we can take $\tau_\alpha, \tau_\gamma, \tau_\eta$ in Theorem \ref{thm:rdgd} to achieve that for any $i \in [n]$, any $\epsilon$ with $0 < \epsilon < 1/3$, almost surely, 
\begin{align*}
    \lim_{t \rightarrow \infty}(t+1)^{1/3-\epsilon}\big(f(\x_i^t) - f(\x^*)\big) = 0.
\end{align*}
\end{corollary}

\begin{remark}
\label{rm:rate}
Theorem \ref{thm:rdgd} states that asymptotically, the algorithm iterates $\x_i^t$ of any agent $i$ almost surely converges to the exact minimum of $f$. The convergence rate is sublinear and depends on $\tau_\alpha, \tau_\gamma$. To obtain Corollary \ref{cor:funcsub} from Theorem \ref{thm:rdgd} we solve a linear program involving $\tau_\alpha, \tau_\gamma$ to maximize $\min(\tau_\gamma, \tau_\alpha/3 - 2\tau_\gamma/3)$ under the constraints $2\tau_\gamma \le \tau_\alpha \le \min(1, 1 - \tau_\gamma) $, which leads to the optimal assignment $\tau_\alpha = 5/6, \tau_\gamma = 1/6$. Thus, we can achieve a convergence rate that is arbitrarily close to $\O(t^{-1/6}) $ for $\norm{\x_i^t - \x^*}$. Since $f$ is also $L$-smooth, using a standard descent lemma \cite{beck2017first}, we have $\forall \x, \y \in \R^d$, 
\begin{align*}
    f(\y) \le f(\x) + \inpd{\nabla f(\x), \y - \x} + \frac{L}{2}\norm{\y - \x}^2.
\end{align*}
Taking $\x = \x^*$ and using the fact that $\nabla f(\x^*)  = 0$ that follows from Assumption \ref{as:cvx}-\ref{as:zero-gd}, we have $f(\x_i^t) - f(\x^*) \le (L/2)\norm{\x_i^t - \x^*}^2$, and thus the convergence rate for $f(\x_i^t) - f(\x^*)$ can be arbitrarily close to $\O(t^{-1/3})$. Note that our convergence rate is established in the almost sure sense, while most existing works in distributed optimization prove mean square convergence, with the exception of \cite{xin2020variance} (see also references therein). The almost sure convergence is beneficial in that it provides convergence guarantee for nearly every algorithm sample path instance. 

Note that \cvrg\ does pay a price in convergence rate for the sake of security. In a one machine setup, the almost sure convergence rate of stochastic gradient descent for strongly convex and smooth objective function can be arbitrary close to $\O(t^{-1})$ \cite{sebbouh2021almost}. In the distributed setup with no gradient attack, \cite{xin2020variance} also proves almost sure convergence arbitrarily close to $\O(t^{-1})$ for distributed stochastic gradient with the aid of gradient tracking.

Finally, we point out some technical differences with respect to \cite{chen2019resilient} that also uses decaying thresholds to achieve perfect recovery in distributed estimation. A key difference in our construction is the use of an additional gradient estimator (locally at each agent) based on instantaneous stochastic gradients. Intuitively, these gradient estimators lead to asymptotically decaying gradient variance which is required for convergence as the clipping operation induces additional nonlinearities. By leveraging the special form of the linear regression type cost models studied in \cite{chen2019resilient}, this step was essentially bypassed in lieu of a simple arithmetic mean of gradients which is not applicable in the current scenario. The addition of the non-trivial gradient estimators in turn necessitate different choices of the algorithm parameters (weight sequences and thresholds) and hence new proof techniques. However, although \cvrg\ applies to more general convex models, we point out that the SAGE algorithm developed in~\cite{chen2019resilient} is optimized for linear regression models and for such setups yields a $\O(t^{-1/4})$ almost sure convergence rate of the agent iterates to the true minimizer, which is better than the $\O(t^{-1/6})$ rate we obtain for more general convex models in the adversarial setting.
\end{remark}

        \section{Experiments}
\label{sec:experiments}
We focus on comparing the performance of \cvrg\ with distributed stochastic gradient descent (DSGD), the most common construction employed in non-adversarial settings. The DSGD algorithm we implement is adapted from the diffusion variant studied in \cite{chen2013distributed}, i.e., each agent in parallel updates its local estimate $\x_i^t$ as follows,
\begin{align}
\label{eq:dsgd}
    \x_i^{t+1} = \sum_{j=1}^n w_{ij} (\x_j^t - \alpha_t \m_j^t).
\end{align}
Most recent resilient distributed optimization algorithms are developed for networks with Byzantine agents that has very different adversarial behaviors in communications and computations with agents under our attack model, so we narrow down our comparisons only with DSGD.

\subsection{Distributed Heterogeneous Measurements}
In this synthetic example, we demonstrate the effectiveness of \cvrg\ in a distributed heterogeneous measurement model. We consider an undirected 2D grid network that consists of 625 agents as shown in Fig \ref{fig:grid_network}, and as indicated by inter-agent links, each agent can only communicate with its direct neighbors or agents at its diagonal position. In Fig \ref{fig:grid_network}, green nodes represent regular agents in $\N$ while black nodes are agents in $\A$ that have adversarial stochastic gradient oracles (arbitrarily corrupted sensor measurements). The network of agents aims to estimate a long vector of 625 environment parameters $\bth_*$ with each component of $\bth_*$ corresponding to the true scalar environment parameter at the location of each agent. However, each agent only has noisy measurements on environment parameters at positions within distance of 5 units (the side length of each cell in the grid is 1 unit), so these agents need to collaborate to estimate $\bth_*$ that has network-wide information. Specifically, for each agent $i \in \N$, we consider the following measurement model,
\begin{align*}
    \y_i^t = \bH_i\bth_* + \w_i^t, 
\end{align*}
where each row of measurement matrix $\bH_i$ is a canonical basis vector of length 625 that measures one component of $\bth_*$ and $\{\w_i^t\}_{t \ge 0}$ are i.i.d. zero mean Gaussian noises. The sensing matrix $\bH_i$ is defined to enable agent $i$ to measure all components of $\bth_*$ that are in positions within distance of 5 units from agent $i$. For example, for agent $i$ at the center of the grid, $\bH_i$ has 46 rows, but if agent $i$ is at the corner of the grid then $\bH_i$ has 26 rows.
\begin{figure}[ht]
\centering
\captionsetup{justification=centering}
 \includegraphics[width=0.35\textwidth]{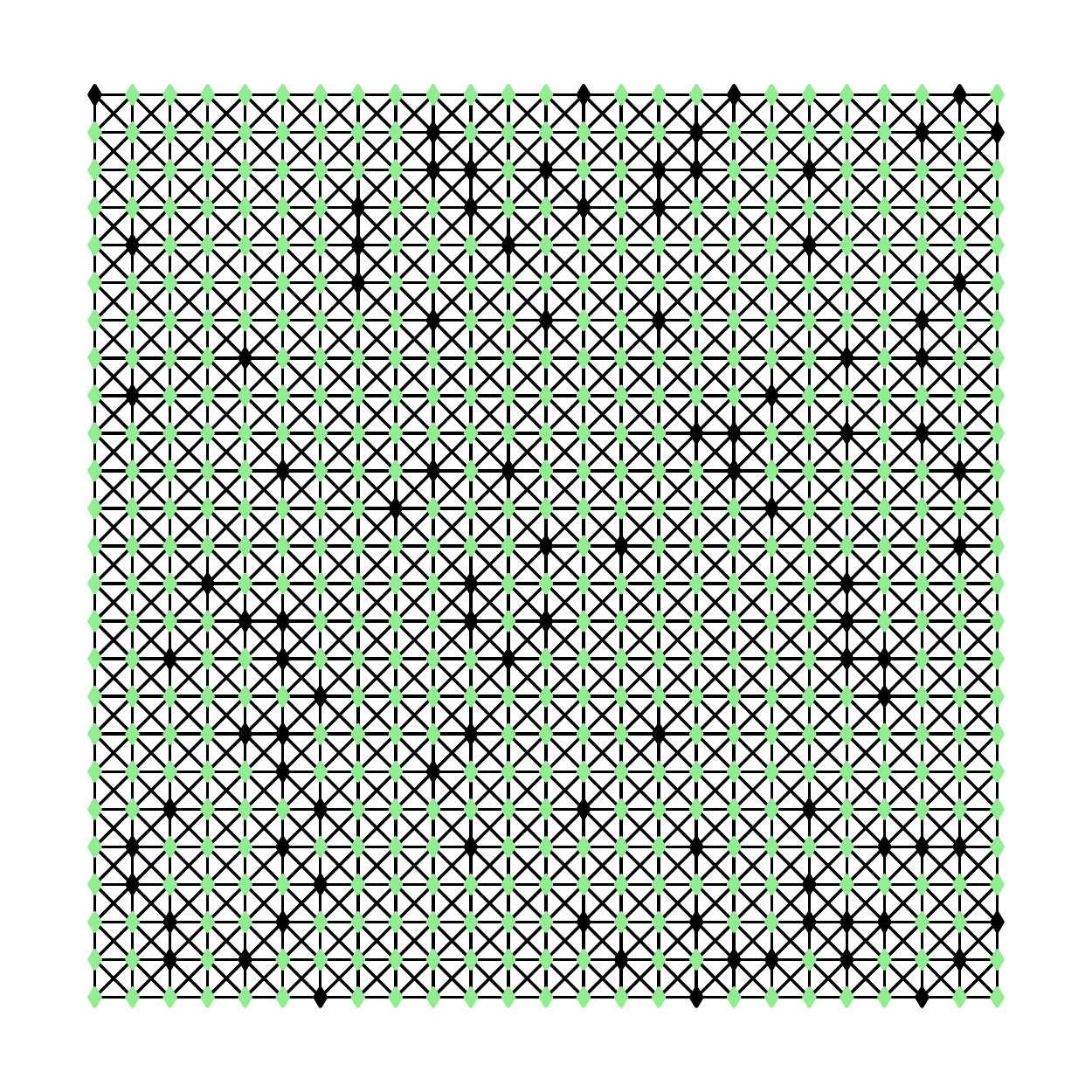}
\caption{2D grid network of 625 agents}
\label{fig:grid_network}
\end{figure} 
To recover the true parameter vector $\bth_*$, we formulate an $\ell_2$ loss minimization problem over regular agents $\N$, 
\begin{align}
\label{eq:recover_obj}
    \text{minimize}_{\x \in \R^{625}} \sum_{i \in \N} \E_{\w_i} \norm{\bH_i \x - \y_i}^2. 
\end{align}
In this example, as shown in Fig \ref{fig:grid_network}, we randomly sample 100 agents that are under gradient attack. In our setting, each local objective function 
\begin{align*}
    f_i = \E_{\w_i} \| \bH_i \x  - \y_i \|^2
\end{align*}
is convex and smooth, the true aggregated objective function $\sum_{i\in \N} f_i$ is strongly convex 
with condition number $\kappa \approx 4.35$, and $\bth_*$ is an optimal solution for each $f_i$. Per Assumption \ref{as:resilience-ratio}, the fraction of attacked agents in this experiment is less than $625/(1 + \kappa) \approx 117$. We sample $\bth_*$ from $[-40, 180]^{625}$ (range of temperature), and design the variance of Gaussian noise as 10.  In our formulation, stochastic gradients only have nonzero entries corresponding to components being measured. For example, the stochastic gradients computed by an agent at the corner of the grid only have 26 nonzero components. To simulate gradient attack, for agents in $\A$, we set the nonzero components of their gradients as $-200$ persistently.

We carefully tuned both DSGD and \cvrg. For DSGD, we pick $\alpha_t = 22/(t+1)$, and for \cvrg\ we choose
$\alpha_t = 220(t+1)^{-0.82}, \gamma_t = 600(t+1)^{-0.17}, \eta_t = 7(t+1)^{-0.66}$
and both algorithms are initialized from the zero vector. For both implementations, we use Metropolis weights \cite{xiao2005scheme} as mixing matrix $\W$. We compare the maximum $\ell_2$ estimation error $(1/625)\max_{i\in[n]}\| \x_i^t - \bth_*\|_2$ of DSGD and \cvrg\ in Fig \ref{fig:perf_linear}. Under the considered gradient attack, DSGD fails to converge to the true parameter $\bth_*$ and even diverge after some iterations while \cvrg\ resiliently minimize the $\ell_2$ error towards 0.
\begin{figure}[ht]
\centering
\includegraphics[width=0.45\textwidth]{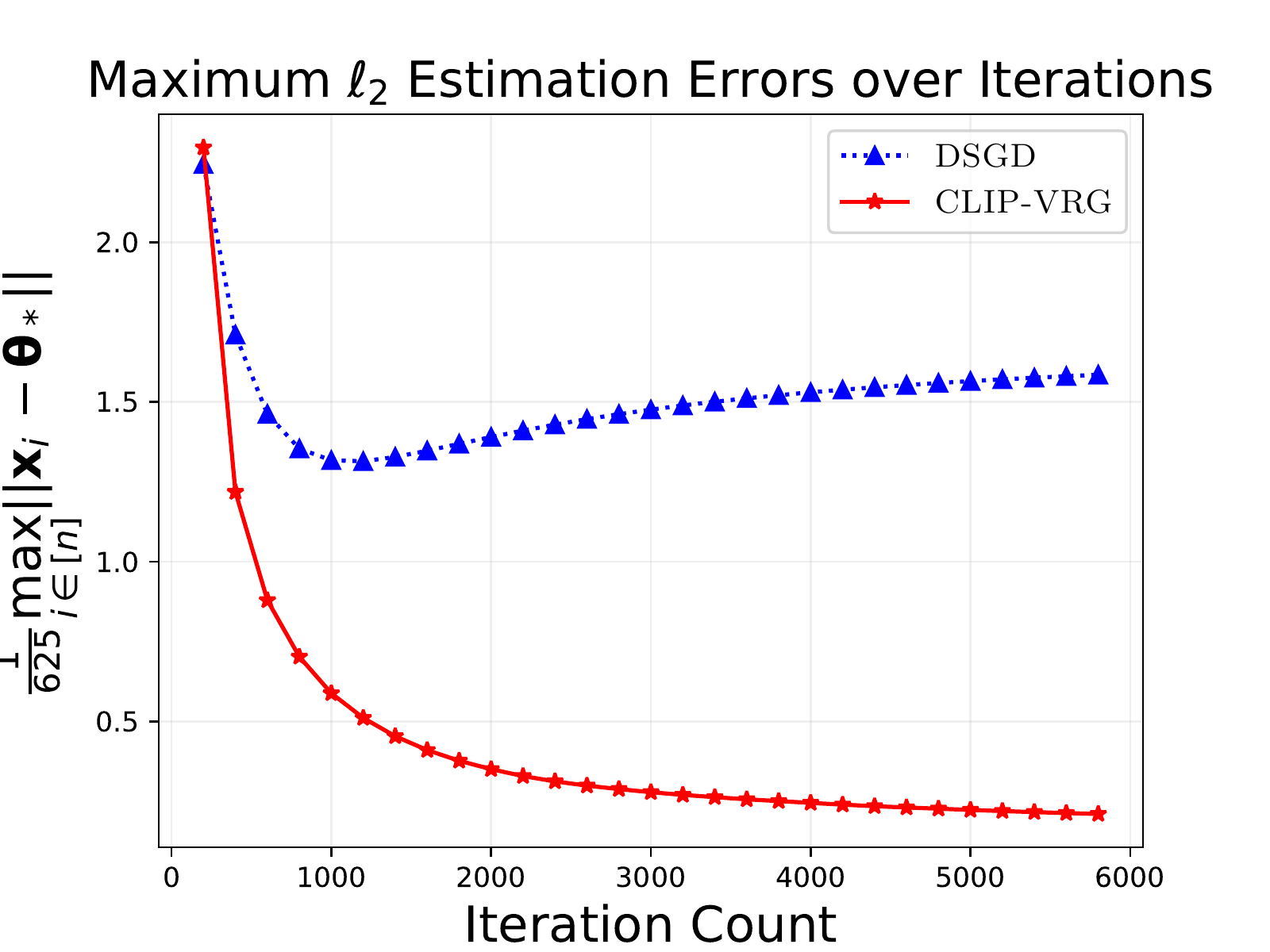}
\caption{Maximum $\ell_2$ estimation errors comparison}
\label{fig:perf_linear}
\end{figure}

\subsection{Distributed Binary Classification}
In this experiment we use real-world image classification datasets to test the efficacy of \cvrg. We study a scenario where each networked agent solves the same empirical risk minimization formulation for a binary classification task to simulate distributed learning or inference on homogeneous data. Specifically, each agent has the same dataset $\{\bth_i, \xi_i\}$ and tries to minimize a regularized logistic regression objective
\begin{align*}
    \ell(\x, \{\bth_i, \xi_i\}_{i = 1, \ldots, n}) = \frac{1}{n} \sum_{i=1}^n \ln\big(1 + e^{-\x^\top\bth_i   \xi_i}\big) + \frac{\lambda}{2}\|\x\|_2^2,
\end{align*}
where $\bth_i$ denotes the feature vector of $i$th data point, $\xi_i$ is its corresponding binary label in $\{-1, 1\}$, and $\lambda$ is a regularization parameter to control overfitting. We perform this experiment on two graph topologies with different datasets. 

In our first experiment setup, we consider an undirected geometric random graph of 100 agents among which 25 agents are under arbitrary gradient attack (See Fig \ref{fig:geo_graph_mnist}, black nodes represent attacked agents). Each agent has access to the same Fashion-MNIST dataset \cite{chang2011libsvm}, and use data points with labels "pullover" and "coat". Each label has 5000 training data points and 2000 test data points, and each agent solves the same regularized logistic regression formulation with $\lambda=0.1$ on training data. For regular agents in $\N$, mini-batch stochastic gradients are sampled at each iteration, but for attacked agents in $\A$, at each iteration their gradient oracles persistently return $c\one_{784}$ for constant $c \approx 0.714$. In this problem setting, we can estimate the upper bound on the fraction of attacked agents in Assumption \ref{as:resilience-ratio} is $1/(1 + \kappa) \ge 0.26$, which is larger than the actual fraction of the attacked agents $0.25$. Note that Assumptions \ref{as:smooth}-\ref{as:resilience-ratio} are all satisfied in this experiment setup, except that Assumption \ref{as:noise} is not necessarily satisfied since regularized logistic regression may have unbounded gradient when $\x$ is unbounded. In Figure \ref{fig:geo_graph_mnist}, we compare the performance of DSGD and \cvrg\ in terms of averaged test accuracy on test data points and the averaged optimality gap of training loss, i.e., $\ell(\x^t, \{\bth_i, \xi_i\}_{i = 1, \ldots, n}) - \ell(\x_*, \{\bth_i, \xi_i\}_{i = 1, \ldots, n})$, over all agents. Under the aforementioned attack, DSGD fails to reach the optimal solution with high residual errors and fails to obtain a workable classifier. In contrast, \cvrg\ resiliently optimizes the loss function with a minor sacrifice in the convergence speed and final optimality precision, and the test accuracy is as high as in unattacked cases. Moreover, in the case without attack, \cvrg\ shows no comprise in performance to achieve security guarantees at the same time.
\begin{figure*}
\centering
\includegraphics[width=0.325\textwidth]{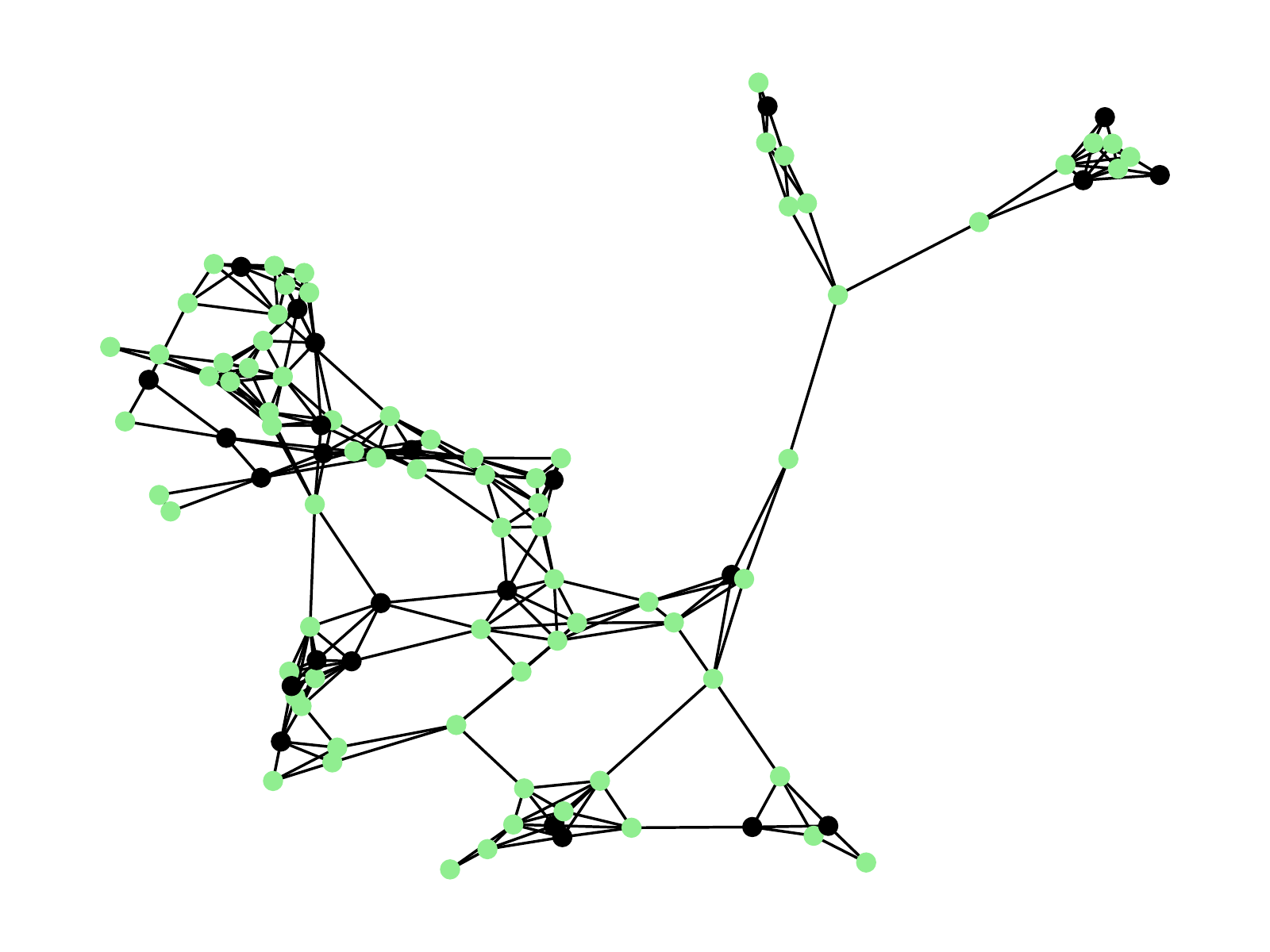}
\includegraphics[width=0.325\textwidth]{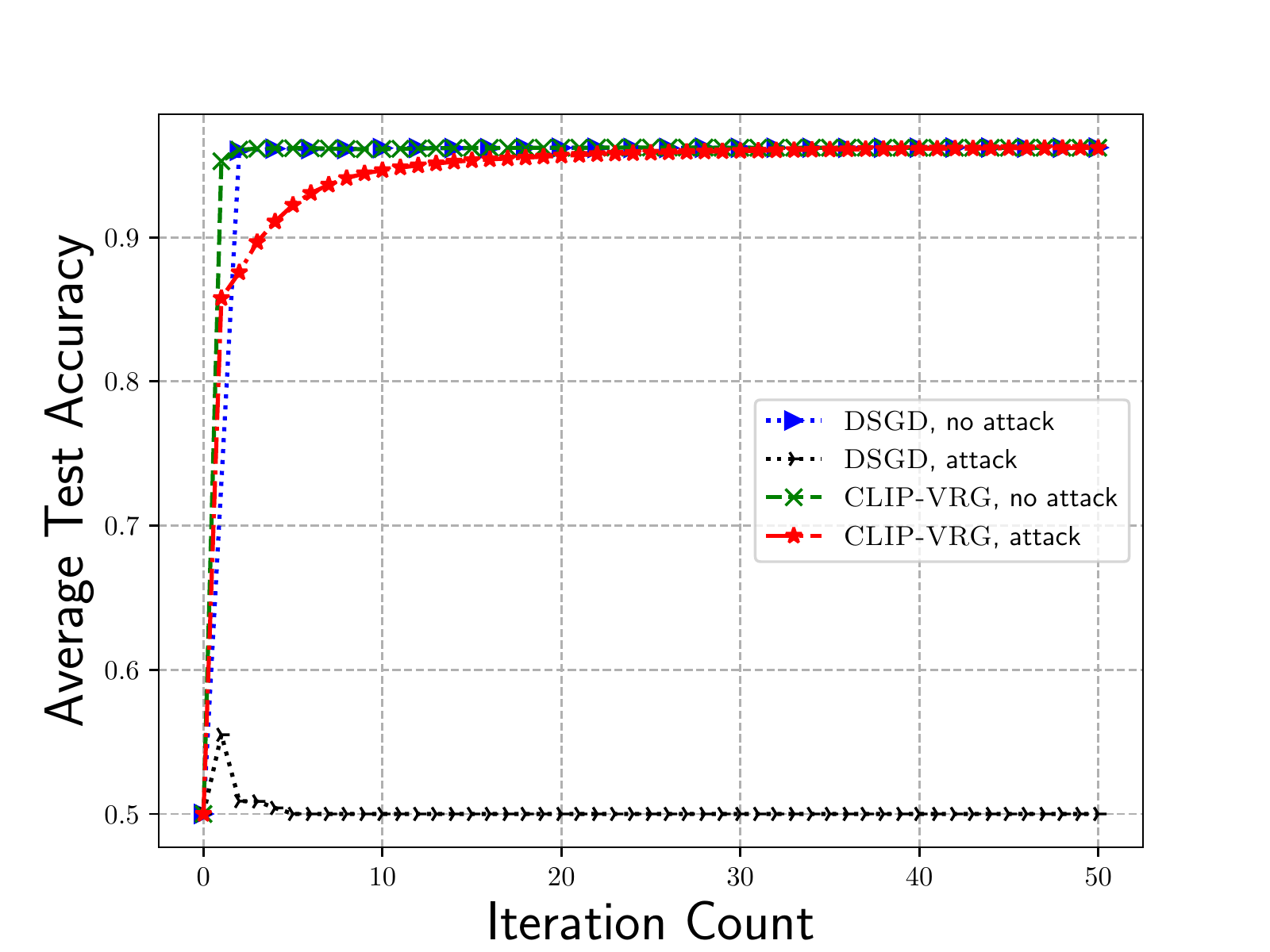}
\includegraphics[width=0.325\textwidth]{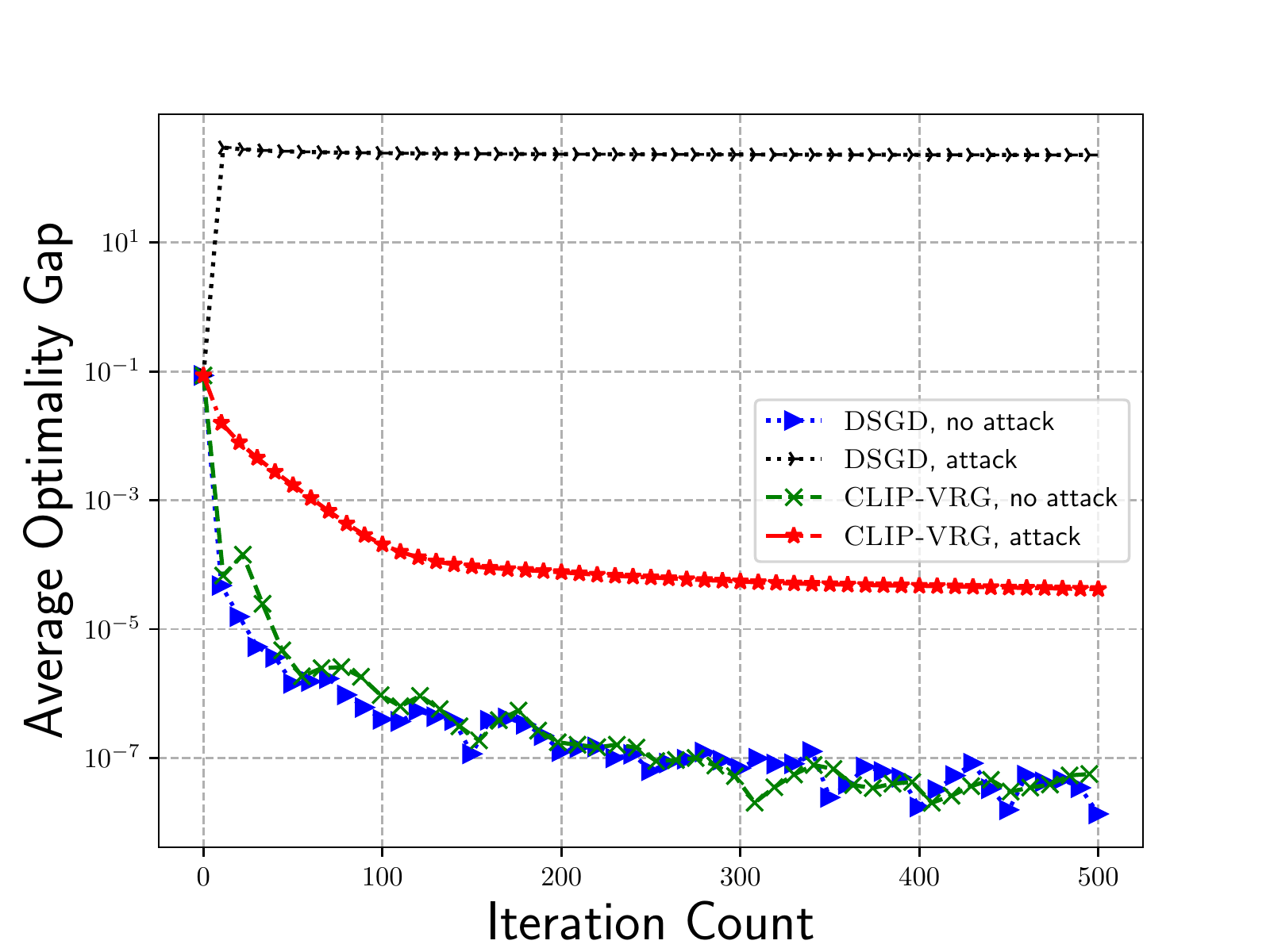}
\caption{An undirected random geometric graph of 100 agents. Performance comparison of DSGD and \cvrg\ in case of no gradient attack and persistent gradient attack on Fashion-MNIST dataset.} 
\label{fig:geo_graph_mnist}
\end{figure*}

In our second experiment setup, we consider a simpler topology, a connected cycle of 15 agents where each agent has 6 neighbors, and 3 agents are under arbitrary gradient attack (see Fig \ref{fig:cycle_graph_cifar10}, black agents represent attacked agents). Each agent has access to the same CIFAR-10 \cite{chang2011libsvm} dataset, which is harder to train for a logistic regression model compared to Fashion-MNIST. Similarly, we choose data points with labels "cat" and "dog" to define the regularized logistic regression loss in training data points, and compare the performance of DSGD and \cvrg\ in terms of averaged test accuracy and averaged optimality gap of training loss. In this setup, we choose regularization parameter $\lambda = 10^{-3}$, and each attacked agent receives persistent gradient $c\one_{3072}$ for constant $c \approx 0.18$. In Fig \ref{fig:cycle_graph_cifar10}, under gradient attack, \cvrg\ can still resiliently optimize the loss function towards optimal with a slower convergence speed, and achieves comparable test accuracy with unattacked case after some more iterations. In contrast, DSGD fails to make progresses towards optimizing the loss and training a meaningful classifier. However, in the case without gradient attack, \cvrg\ does pay a minor price in performance to achieve security guarantees. This validates our discussions in Remark \ref{rm:rate} that the best achievable almost sure convergence rate of \cvrg\ is inferior with respect to the theoretically achievable in non-adversarial environments. (In this experiment, we use the optimized $\tau_\alpha, \tau_\gamma$ in Remark \ref{rm:rate}.)

\begin{figure*}
\centering
\includegraphics[width=0.325\textwidth]{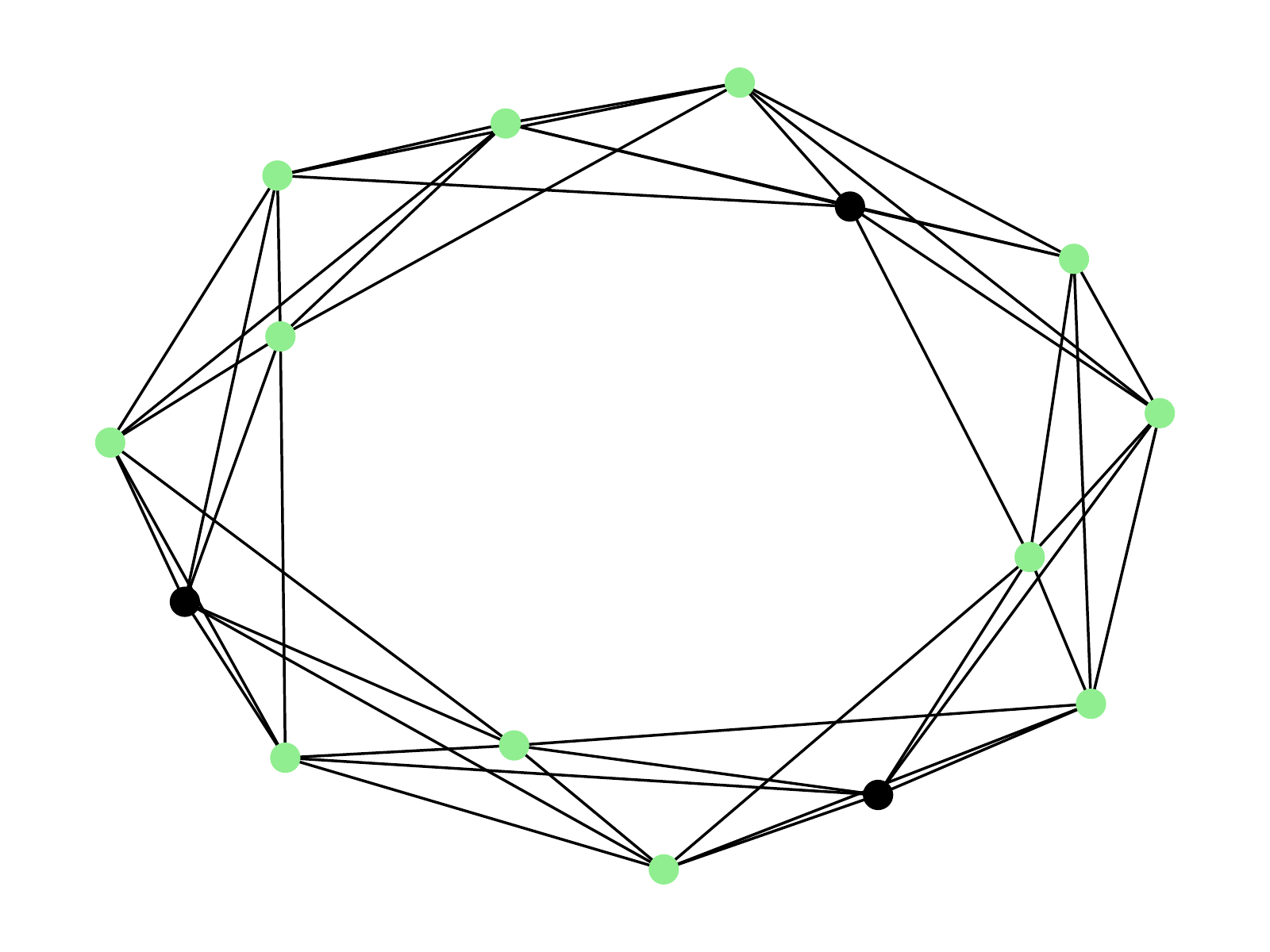}
\includegraphics[width=0.325\textwidth]{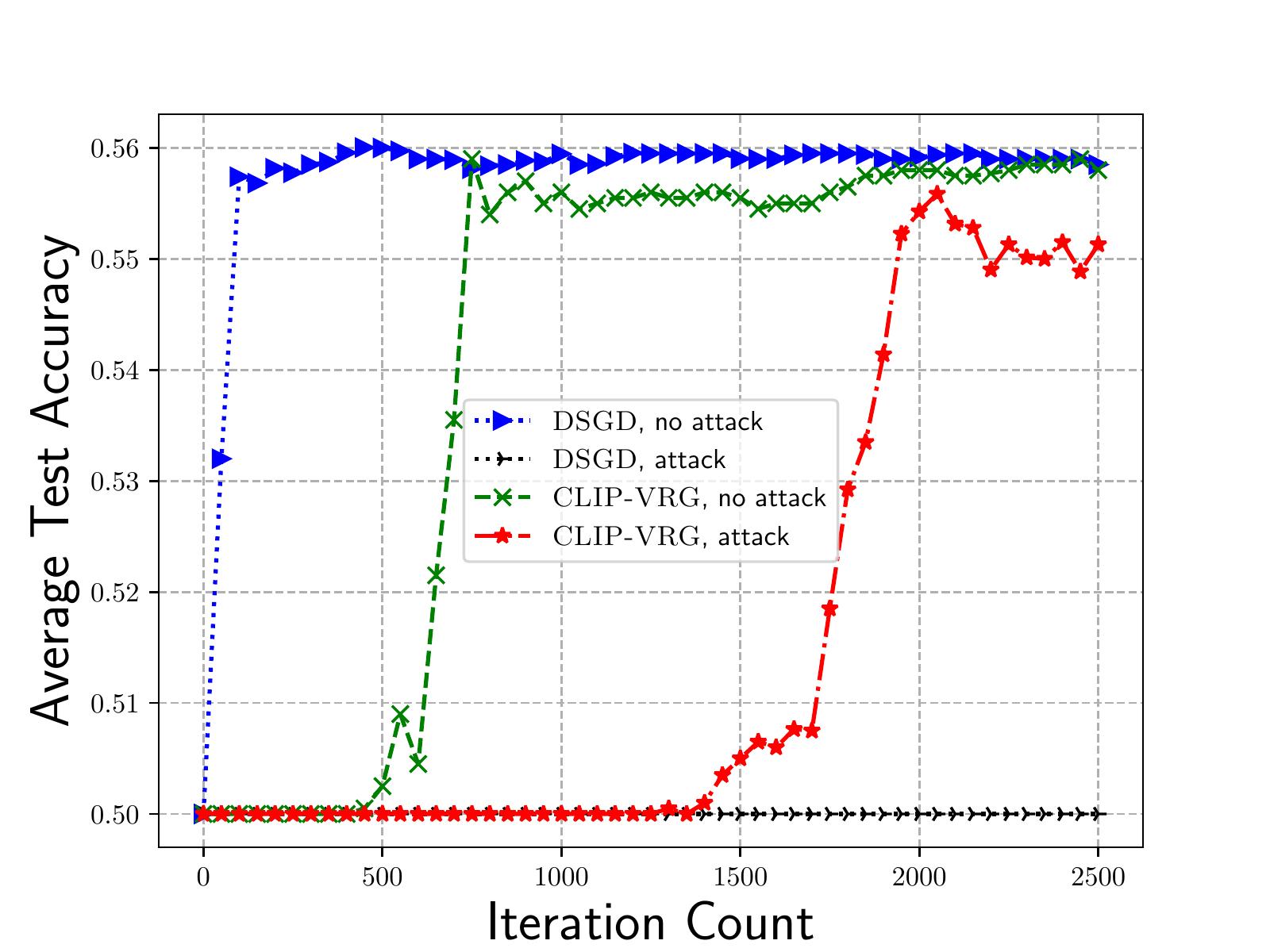}
\includegraphics[width=0.325\textwidth]{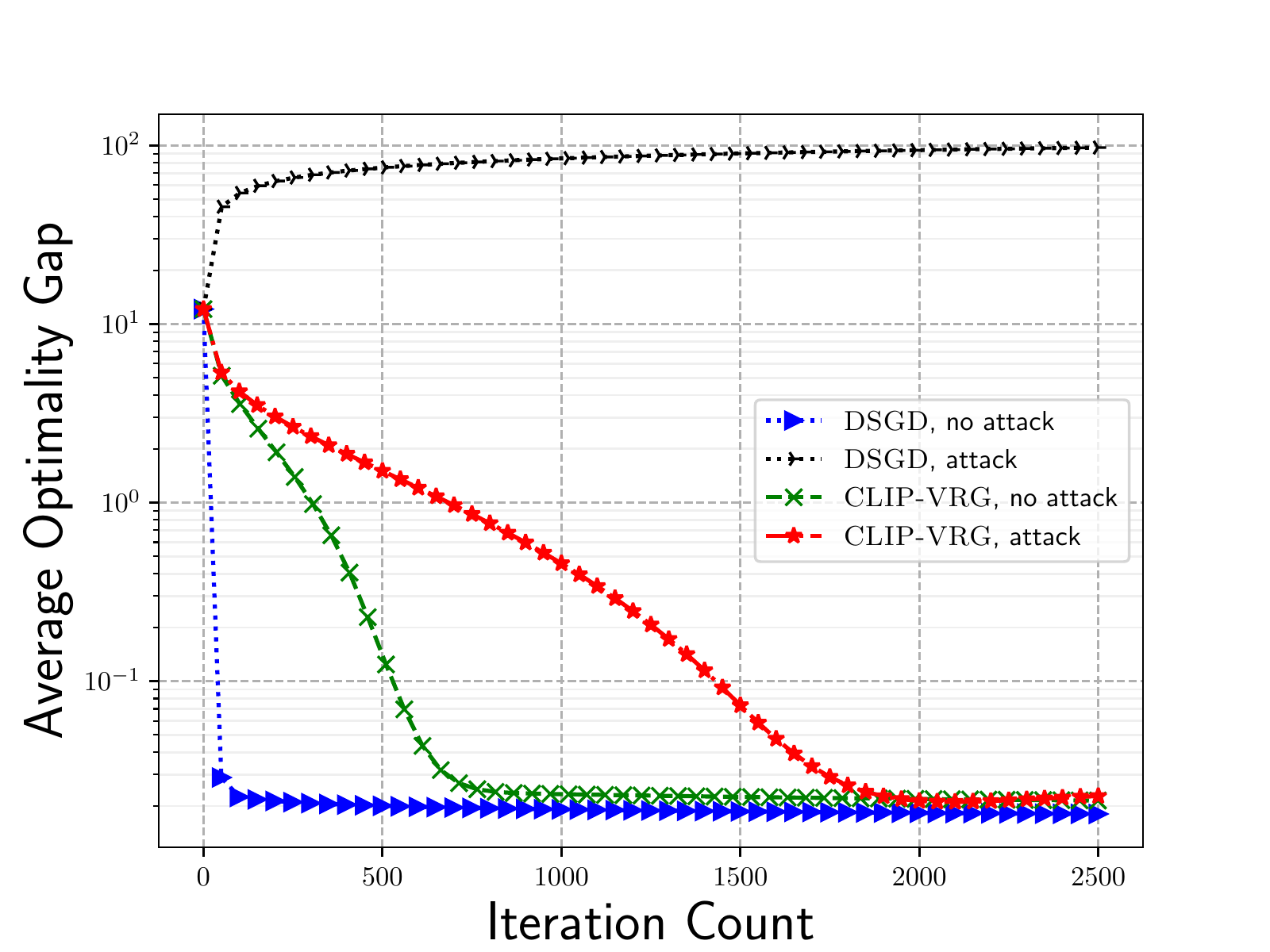}
\caption{An connected cycle of 15 agents. Performance comparison of DSGD and \cvrg\ in case of no gradient attack and persistent gradient attack on CIFAR-10 dataset.}
\label{fig:cycle_graph_cifar10}
\end{figure*}

        \section{Proof of Theorem \ref{thm:rdgd}}
\label{sec:proofs}
Define the network average of all local decision variables at iteration $t$ as $\ox^{t} := (1/n) \sum_{i=1}^n \x_i^t$. Then, by the double stochasticity of $\W$, averaging \eqref{eq:local-rdgd} over all $i \in [n]$ gives that
\begin{align}
\label{eq:gdmap-update}
     \ox^{t+1} = \ox^{t} - \frac{\alpha_t}{n}\sum_{i = 1}^n k_i^t \bv_i^t.
\end{align}
Define long vectors and diagonal matrix
\begin{align*}
    & \x^t = 
    \begin{bmatrix}
    \x_1^t\\
    \vdots\\
    \x_n^t
    \end{bmatrix}
    , 
    \bv^t = 
    \begin{bmatrix}
    \bv_1^t\\
    \vdots\\
    \bv_n^t
    \end{bmatrix}
    , \\
    & \K^t = 
    \text{diag}([k_1^t\one^\top_d, \ldots, k_n^t \one^\top_d]). 
\end{align*}
Then, all local updates at iteration $t$ can be summarized as
\begin{align}
\label{eq:alg-stack}
    \x^{t+1} = (\W \otimes \I_d)(\x^t - \alpha_t \K^t \bv^t).
\end{align}

We first develop the following lemma to estimate consensus error $\|\x^t - \one_n \otimes \ox^t\|$, i.e., the distance between local variables $\x_i^t$ and network average $\ox^t$.
\begin{lemma}
\label{lm:consensus}
Take integer $\varphi > 1/\big(1 - \beta^{1/(\tau_\alpha + \tau_\gamma)}\big) - 1$. Then, the iterate $\x^t$ generated by \cvrg\ satisfies that for any constant $c \ge \beta / \big[ \big(\varphi / (1 + \varphi)\big)^{\tau_\alpha + \tau_\gamma} - \beta \big]$, we have
\begin{align}
\label{eq:consensus-error-1}
    \|\x^t - \one_n \otimes \ox^t \| \le \sqrt{n}  \sum_{s=0}^{t-1} \beta^{t-s} \alpha_s \gamma_s \le c \sqrt{n}\alpha_t \gamma_t. 
\end{align}
\end{lemma}
\begin{proof} 
By the definition of $k_i^t$, we have
\begin{align}
\label{eq:d-bound}
    \forall i \in [n], \ \norm{ k_i^t \bv_i^t} \le \gamma_t. 
\end{align}
From~\eqref{eq:alg-stack} we have
\begin{align}
\label{eq:xt}
    \x^t = (\W \otimes \I_d)^t \x^0 - \sum_{s = 0}^{t-1}  (\W \otimes \I_d)^{t-s} \alpha_s \K^s \bv^s.
\end{align}
Then, 
\begin{align*}
\begin{split}
& \|\x^t - \mathbf{1}_n \otimes \ox^t\| \\
= \ & \norm{ \big( \I_{nd} - \frac{1}{n}\mathbf{1}_n \mathbf{1}_n^{\top} \otimes \I_d \big)\x^t} \\
= \ & \norm{ \big(\frac{1}{n} \mathbf{1}_n \mathbf{1}_n^\top \otimes \I_d - \I_{nd} \big)\sum_{s = 0}^{t - 1}(\W \otimes \I_d)^{t-s} \alpha_s \K^s \bv^s} \\
\le \ &  \sum_{s=0}^{t-1} \norm{  \frac{1}{n}\mathbf{1}_n \mathbf{1}_n^\top - \W^{t-s} }_2 \norm{\alpha_s \K^s \bv^s} \\
\overset{\eqref{eq:d-bound}}{\le} \  & \sqrt{n}  \sum_{s=0}^{t-1} \beta^{t-s} \alpha_s \gamma_s.
\end{split}
\end{align*}
In the second equality above, we exploited the fact that
\begin{align}
\label{eq:initial}
    \Big( \I_{nd} - \frac{1}{n}\mathbf{1}_n \mathbf{1}_n^{\top} \otimes \I_d \Big) (\W \otimes \I_d)^{t} \x^0 = 0,
\end{align}
owing to the initialization $\x_i^0 = \x_j^0$ for all $i, j \in [n]$ and the assumption that $\W$ is doubly stochastic. In the first inequality above, we used the Assumption \ref{as:mixmatrix} that $\W$ is real symmetric so $\W^{t-s}$ has eigenvalues $(\lambda_i(\W))^{t-s}$ for $i = 1, \ldots, n$, and $\W$ is stochastic thus the eigenvalue of $\W$ associated with 1 is $(1/\sqrt{n})\one_n$. We next show that for some $c := c(\beta,  \tau_\alpha, \tau_\gamma, \varphi)$, we have 
\begin{align}
\label{eq:sum-bd}
    \sum_{s = 0}^{t - 1} \beta^{t - s} \alpha_s \gamma_s \le c \alpha_t \gamma_t.  
\end{align}
For $t = 1$, it suffices to have 
\begin{align}
\label{eq:cns-base}
    \beta \alpha_0 \gamma_0 \le c \alpha_1 \gamma_1,
\end{align}
which is equivalent to
\begin{align}
\label{eq:ccond1}
    c \ge \beta (1 + \frac{1}{\varphi})^{\tau_\alpha + \tau_\gamma}. 
\end{align}
Suppose \eqref{eq:sum-bd} holds true for some $t \ge 1$, to ensure \eqref{eq:sum-bd} for $t + 1$, it suffices to have 
\begin{align*}
     \sum_{s = 0}^t \beta^{t + 1 - s} \alpha_s \gamma_s = \beta \sum_{s = 0}^{t - 1} \beta^{t - s} \alpha_s \gamma_s + \beta \alpha_t \gamma_t \le c \alpha_{t+1} \gamma_{t+1}. 
\end{align*}
Substituting \eqref{eq:sum-bd} for $t$ to the above relation, it suffices to have 
\begin{align}
\label{eq:induc_reduced}
    c(\alpha_{t+1}\gamma_{t+1} - \beta\alpha_{t}\gamma_t) \ge \beta \alpha_t \gamma_t. 
\end{align}
We take positive $\varphi$ such that $\alpha_{t+1}\gamma_{t+1} > \beta\alpha_{t}\gamma_t$ for all $t \ge 0$, which is equivalent to  
\begin{align}
\label{eq:phicond}
    \varphi > \frac{1}{1 - \beta^{1/(\tau_\alpha + \tau_\gamma)}} - 1.  
\end{align}
Then, \eqref{eq:induc_reduced} boils down to 
\begin{align}
\label{eq:ccond2}
    c \ge \frac{\beta}{(\frac{\varphi}{1 + \varphi})^{\tau_\alpha + \tau_\gamma} - \beta}.
\end{align}
Taking positive $c, \varphi$ that satisfy \eqref{eq:ccond1}, \eqref{eq:phicond}-\eqref{eq:ccond2}, and combing with the base case \eqref{eq:cns-base} finish the proof of this lemma. 
\end{proof}


We next try to upper bound the gradient estimation errors on regular agents. As an intermediate result, we first show the following lemma. 
\begin{lemma}
\label{lm:mse2as}
Let $\{z_t\}$ be an $\R_{+}$ stochastic sequence. Let $\G_{t +1}$ be the $\sigma$-algebra generated from $\{z_t\}_{t = 1}^k$. Suppose that for some positive constants $c_1, c_2, 0 < a < 1$ and $a < b < a + 1$, $\{z_t\}$ satisfies that
\begin{align}
\label{eq:msecont}
    \E\big(z_{t+1}\mid \G_{t+1}\big) \le (1 - c_1(t+1)^{-a}) z_t + c_2(t+1)^{-b}.
\end{align}
Then, we have that for any $0 < \epsilon_0 < b - a$,
\begin{align*}
    \P\big(\lim_{t \rightarrow \infty} (t+1)^{b - a - \epsilon_0} z_t  = 0 \big) = 1. 
\end{align*}
\end{lemma}
\begin{proof}
The proof is adapted from the proof of Lemma 1 in \cite{yu2021dynamic}. Applying Lemma \ref{lm:cvg_diff_main} in the Appendix leads to that $\forall 0 < \epsilon_0 < b - a$, 
\begin{align}
\label{eq:msecvgzt}
    \lim_{t \rightarrow \infty} (t+1)^{b - a - \epsilon_0}\E\big(z_t\big) = 0.
\end{align}
Now, we fix $\epsilon_0$. Since $0 < b - a - \epsilon_0 < 1$,  $(t+1)^{b - a - \epsilon_0}$ is a concave function of $t$, and thus
\begin{align*}
    (t+2)^{b - a - \epsilon_0} \le (t+1)^{b - a - \epsilon_0}[1 + (b - a - \epsilon_0)(t+1)^{-1}]. 
\end{align*}
Multiplying the both sides of the above relation into the both sides of \eqref{eq:msecont} we obtain that for sufficiently large $t$, there exist some constants $c_3, c_4$ such that
\begin{align}
\begin{aligned}
\label{eq:asytp2}
    & (t+2)^{b - a - \epsilon_0} \E\big(z_{t+1}\mid \G_{t+1}\big) \\
    \le \ & \Big[1 - \frac{c_1}{(t+1)^a} + \frac{b - a - \epsilon_0}{t + 1} - \frac{c_1(b - a - \epsilon_0)}{(t+1)^{a + 1}} \Big]\\
    & \cdot (t+1)^{b - a - \epsilon_0} z_t + \frac{c_2}{(t+1)^{a + \epsilon_0}}\Big(1 + \frac{b - a - \epsilon_0}{t + 1}\Big) \\
    \le \ & \big(1 - \frac{c_3}{(t+1)^{a}}\big)(t+1)^{b - a - \epsilon_0} z_t + \frac{c_4}{(t+1)^{a + \epsilon_0}}.
\end{aligned}
\end{align}
Define the process
\begin{align*}
    V(t) & =  (t+1)^{b - a - \epsilon_0} z_t \\
   & \quad - \sum_{i=0}^{t-1}\Big[ \Big( \Pi_{j = i + 1}^{t - 1}(1 - \frac{c_3}{(j + 1)^{a}  }  \Big) \frac{c_4}{(i + 1)^{a + \epsilon_0}} \Big].
\end{align*}
Using Lemma \ref{lm:sum_zero_limit} in the Appendix we obtain that 
\begin{align}
\label{eq:vttailbd}
    \lim_{t \rightarrow \infty} \sum_{i = 0}^{t - 1}\Big[ \Big( \Pi_{j = i+1}^{t - 1} (1 - \frac{c_3}{(j + 1)^a} \Big) \frac{c_4}{(i+1)^{a + \epsilon_0}} \Big] = 0,
\end{align}
where we used the convention that $\Pi_{j = i+1}^{t-1} (1 - c_3(j+1)^{-a}) = 1$ for $j = i -1$. Also note that we can split
\begin{align*}
    & \sum_{i = 0}^t\Big[\Big( \Pi_{j = i + 1}^t (1 - \frac{c_3}{(j + 1)^a}) \Big) \frac{c_4}{(i + 1)^{a + \epsilon_0}} \Big] \\
    = \ & \Big[ 1 - \frac{c_3}{(t + 1)^a} \Big]\sum_{i = 0}^{t - 1}\Big[ \Big( \Pi_{j = i + 1}^{t - 1} ( 1- \frac{c_3}{(j + 1)^a}) \Big) \frac{c_4}{(i + 1)^{a + \epsilon_0}} \Big] \\
    & + \frac{c_4}{(t+1)^{a + \epsilon_0}}.
\end{align*}
Denote $\H_{k+1}$ the natural filtration of the process $\{(t+1)^{b - a - \epsilon_0}\}$, and note that $V(t)$ is adapted to this filtration. Then, by the independence condition, 
\begin{align*}
    & \E(V(t+1) \mid \H_{t+1}) \\
    = \ & \E\Big((t+2)^{b - a - \epsilon_0} z_{t+1}\mid \H_{t+1} \Big) \\
    & - \sum_{i=0}^{t}\Big[ \Big( \Pi_{j = i + 1}^{t}(1 - \frac{c_3}{(j + 1)^{a}  }  \Big) \frac{c_4}{(i + 1)^{a + \epsilon_0}} \Big] \\
    \underset{\eqref{eq:asytp2}}{\le} \ & \Big[1 - \frac{c_3}{(t+1)^a}\Big](t+1)^{b - a - \epsilon_0} z_t + \frac{c_4}{(t+1)^{a + \epsilon_0}}  \\
    & -\sum_{i=0}^{t}\Big[ \Big( \Pi_{j = i + 1}^{t}(1 - \frac{c_3}{(j + 1)^{a}  }  \Big) \frac{c_4}{(i + 1)^{a + \epsilon_0}} \Big] \\
    = \ & \Big[1 - \frac{c_3}{(t+1)^{a}}\Big] V(t) \\
    \le \ & V(t).
\end{align*}
Thus, $\{V(t)\}$ is a supermartingale. By \eqref{eq:vttailbd}, $V(t)$ is bounded from below. It follows that there exists a finite random variable $V_*$ such that $\P(\lim_{t \rightarrow \infty} V(t) = V_*) = 1$. Thus, with \eqref{eq:vttailbd} we have
\begin{align*}
    \P\big(\lim_{t \rightarrow \infty} (t+1)^{b - a - \epsilon_0} z_t = V_*\big) = 1.
\end{align*}
Then, by Fatou's lemma and \eqref{eq:msecvgzt}, we have 
\begin{align*}
    0 & \le \E\big( \lim_{t \rightarrow \infty} (t+1)^{b - a - \epsilon_0} z_t \big) \\
    & \le \liminf_{t \rightarrow \infty} (t+1)^{b - a - \epsilon_0} \E(z_t) = 0.
\end{align*}
Therefore, we have $\P(\lim_{t \rightarrow \infty} (t+1)^{b - a - \epsilon_0} z_t = 0) = 1$.
\end{proof}

\begin{lemma}
\label{lm:gdest-err}
Take $\tau_\eta = 2(\tau_\alpha + \tau_\gamma)/3 $ and $0< 2\tau_\gamma < \tau_\alpha < 1$. For any $i \in  \N$, for any $0 < \epsilon < \tau_\eta/2$, almost surely, there exists some constant $c_p$ such that $\| \bv_i^t - \nabla f_i(\x_i^t)\| \le c_p(t+1)^{-(0.5\tau_\eta  - \epsilon)}$.

\end{lemma}

\begin{proof}
First, we bound the one step difference $\|\x_i^{t+1} - \x_i^t\|^2$. By Jensen's inequality, 
\begin{align*}
&  \norm{\x_i^{t+1} - \x_i^t}^2 \\
= \ & \norm{\sum_{j=1}^n w_{ij}(\x_j^t - \alpha_t k_j^t \bv_j^t)  - \x_i^t}^2 \\
\le \  &  2 \sum_{j = 1}^n w_{ij} \big( \norm{\x_j^t - \x_i^t}^2 + \alpha_t^2 \norm{k_j^t \bv_j^t}^2 \big). 
\end{align*}
By Lemma \ref{lm:consensus}, for any $i, j \in [n]$ and $i \ne j$, 
\begin{align*}
    \norm{\x_j^t - \x_i^t}^2 \le 2 \norm{\x_j^t - \ox^t}^2 + 2 \norm{\x_i^t - \ox^t}^2 \le 2 n c^2 \alpha_t^2 \gamma_t^2.
\end{align*}
Together with \eqref{eq:d-bound}, we obtain 
\begin{align*}
    \norm{\x_i^{t+1} - \x_i^t}^2 \le (4 n c^2 +2) \alpha_t^2 \gamma_t^2.   
\end{align*}
Next, we establish the recursion for gradient estimation errors. From \eqref{eq:eta_vr}, we have
\begin{align}
\begin{aligned}
\label{eq:pt-decom}
    & \bv_i^{t+1} - \nabla f_i(\x_i^{t+1}) \\
     = \ & (1 - \eta_t)(\bv_i^t - \nabla f_i(\x_i^t)) \\
     & +  (1 - \eta_t)(\nabla f_i(\x_i^t) - \nabla f_i(\x_i^{t+1}))
     + \eta_t \bxi_i^{t+1}. 
\end{aligned}
\end{align}
Let $\F_{t+1}$ be the  $\sigma$-algebra generated by $ \{ \x_i^0, \{\bxi_i^s\}_{i \in [n], 0 \le s \le t}\}$. Define $\p_i^t := \bv_i^t - \nabla f_i(\x_i^{t})$. By \eqref{eq:pt-decom} and the assumptions on the gradient noises on good agents, we have, for $t \ge 1$,
\begin{align*}
    & \E(\| \p_i^{t+1} \|^2 \mid \F_{t+1}) \\
    \le \ &  
     (1 - \eta_t)^2 \norm{\p_i^t}^2 + \eta_t^2 \E( \norm{\bxi_i^{t+1}}^2 \mid \F_{t+1})  \\
    &  + (1 - \eta_t)^2 \E(\norm{\nabla f_i(\x_i^t) - \nabla f_i(\x_i^{t+1})}^2\mid \F_{t+1})  \\
    &  + 2(1 - \eta_t)^2 \inpd{\p_i^t, \nabla f_i(\x_i^t) - \nabla f_i(\x_i^{t+1})} \\
    \le \ &  (1 - \eta_t)^2 \norm{\p_i^t}^2 + \eta_t^2 \sigma^2 \\
    & + (1 - \eta_t)^2 L^2 (4nc^2 + 2)\alpha_t^2 \gamma_t^2  + (1 - \eta_t)^2 \\
    & \cdot \Big [\frac{\eta_t}{2} \norm{\p_i^t}^2   + \frac{2}{\eta_t} \E\big( \norm{\nabla f_i(\x_i^t) - \nabla f_i(\x_i^{t+1})}^2\mid \F_{t+1} \big)\Big] \\
    \le \ &  (1 - \eta_t)^2(1 + \frac{\eta_t}{2})\norm{\p_i^t}^2 + \eta_t^2 \sigma^2 \\
    & + (1-\eta_t)^2L^2(1 + \frac{2}{\eta_t})(4nc^2 + 2)\alpha_t^2 \gamma_t^2  \\
    \le \ & (1 - \eta_t)\norm{\p_i^t}^2  + \eta_t^2 \sigma^2 + \frac{3 L^2}{\eta_t}(4 nc^2 + 2)\alpha_t^2 \gamma_t^2
\end{align*}
where in the last inequality we used $0 < \eta_t < 1$ and thus
\begin{align*}
    0 < (1 - \eta_t)^2 (1 + \frac{\eta_t}{2}) = 1 - \eta_t  + \frac{1}{2}\eta_t^3 - \frac{1}{2}\eta_t \le 1 - \eta_t. 
\end{align*}
We take $\tau_\eta = 2(\tau_\alpha + \tau_\gamma)/3$. Then, we obtain
\begin{align*}
    & \E\big( \norm{\p_i^{t+1}}^2 \mid \F_{t+1} \big) \\
    \le \ & (1 - \eta_t) \norm{\p_i^t}^2  \\
    & + \Big[c_\eta^2 \sigma^2 + \frac{3L^2}{c_\eta}(4nc^2 + 2)c_\alpha^2c_\gamma^2\Big] (t+\varphi)^{-2\tau_\eta}.
\end{align*}
Then, since $0 < 2\tau_\gamma < \tau_\alpha < 1$, we have $0 < \tau_\eta < 1$. Using Lemma \ref{lm:mse2as}, we obtain that for any $0 < 2\epsilon < \tau_\eta$,
\begin{align*}
    \P\Big(\lim_{t \rightarrow \infty} (t+1)^{\tau_\eta - 2\epsilon} \norm{\p_i^t}^2 = 0\Big) = 1, 
\end{align*}
and thus the lemma follows.
\end{proof}

To show that the network average $\ox^t$ almost surely converges to the minimum $\x^*$, we discuss two exclusive cases. We first show the existence of a local convergence region. Before that, we present a standard result in convex optimization. 

\begin{lemma}
\label{lm:gd-strongcvx}
Suppose function $h: \R^d \rightarrow \R$ is $\mu$-strongly convex and $L$-smooth with minimizer $\x^*$. Then, the iterates generated by gradient descent  $\x' =  \x - \alpha \nabla h(\x)$ with stepsize $0 < \alpha \le 2/(L + \mu)$ satisfy that $\norm{\x' - \x^*} \le (1 - \alpha \mu) \norm{ \x - \x^*}$.
\end{lemma}
\begin{proof}
For completeness, we prove this lemma. Define $p(\x) = h(\x) - (\mu/2)\norm{\x - \x^*}^2$, then $p$ is convex and $(L-\mu)$-smooth. If $L > \mu$, we have (Theorem 2.1.5 in \cite{nesterov2018lectures})
\begin{equation}
\label{eq:cvx-smooth-property}
    \inpd{\nabla p(\x), \x - \x^*} \ge \frac{1}{L - \mu}\norm{\nabla p(\x)}^2.
\end{equation}
Since
\begin{equation}
\label{eq:regd}
    \x' - \x^* = (1 - \alpha \mu)(\x - \x^*) - \alpha \nabla p(\x), 
\end{equation}
we obtain 
\begin{equation*}
\begin{split}
    & \norm{\x' - \x^*}^2 \\
    = \ & (1 - \alpha \mu)^2 \norm{\x - \x^*}^2 - 2\alpha(1 - \alpha \mu)\inpd{\x - \x^*, \nabla p(\x)} \\
    & \quad + \alpha^2 \norm{\nabla p(\x)}^2 \\
    \overset{\eqref{eq:cvx-smooth-property}}{\le} \ & (1 - \alpha \mu)^2 \norm{\x - \x^*}^2 - \frac{\alpha[2 - \alpha(\mu + L)]}{L - \mu} \norm{\nabla p(\x)}^2.
\end{split}
\end{equation*}
Since $\alpha \le 2/(\mu + L)$, the second term of the above display is nonnegative, so the desired relation is obtained. If $L = \mu$, then $h$ is a quadratic function and it turns out
\begin{equation*}
    h(\x) = h(\x^*) + \frac{\mu}{2} \norm{\x - \x^*}^2, 
\end{equation*}
and $\nabla p(\x) = 0$, so \eqref{eq:regd} reduces to $\x' - \x = (1 - \alpha \mu)(\x - \x^*)$, and taking Euclidean norms on both sides completes the proof.
\end{proof}

\begin{lemma}
\label{lm:case1}
Take $\tau_\eta = 2(\tau_\alpha + \tau_\gamma)/3$ and $0 < 2 \tau_\gamma < \tau_\alpha < 1$. Define the auxiliary threshold
\begin{align}
\label{eq:bargt}
    \og_t = \frac{\gamma_t}{L} - \frac{p_t}{L} -  c\sqrt{n}\alpha_t\gamma_t,
\end{align}
where $p_t = c_p(t+1)^{-(0.5\tau_\eta - \epsilon)}$ for some $c_p$ and arbitrary small $0 < \epsilon < 0.5\tau_\eta$. Almost surely, there exist some constant $c_p$, finite $T_0$ such that if for some $t \ge T_0$ we have $\norm{\ox^t - \x^*} \le \og_t$, then $\norm{\ox^t - \x^*} \le \og_t$ for all $t \ge T_0$.  
\end{lemma}
\begin{proof}
Consider the sample path $\omega \in \Omega$ such that for some $c_{p,\o}$ we have for all $i \in \N$, 
\begin{align*}
    \norm{\p_i^{t, \o}} & = 
    \norm{\bv_i^{t, \o} - \nabla f_i(\x_i^{t, \o})} \\
    & \le p_{t, \o} = c_{p, \o} (t+1)^{-(0.5 \tau_\eta - \epsilon)}.
\end{align*}
By Lemma \ref{lm:gdest-err}, such sample paths have probability measure 1. Since $\tau_\alpha > 2 \tau_\gamma$, we have
\begin{align*}
    \tau_\gamma < (\tau_\alpha + \tau_\gamma)/3 - \epsilon = 0.5 \tau_\eta - \epsilon, 
\end{align*}
for arbitrarily small $\epsilon$. Then, in \eqref{eq:bargt}, $\gamma_t$ decays slower than $p_{t, \o}$ and $\alpha_t \gamma_t$, and thus there exists some finite $t_1$ such that $\forall t \ge t_{1, \o}$, $\og_{t, \o} > 0$. We decompose the update of network average $\ox^{t, \o}$, 
\begin{align}
\label{eq:avg-update}
\begin{aligned}
    \ox^{t+1, \o} 
    & = \ox^{t, \o} - \frac{\alpha_t}{n}\sum_{i=1}^n k_i^{t, \o}\bv_i^{t, \o} \\
    & = \ox^{t, \o} - \frac{\alpha_t}{n} \sum_{i \in \N} k_i^{t, \o} (\nabla f_i(\x_i^{t,\o}) + \p_i^{t, \o}) \\
    & \qquad - \frac{\alpha_t}{n} \sum_{i \in \A} k_i^{t, \o} \bv_i^{t, \o}. 
\end{aligned}
\end{align}
By Assumption \ref{as:zero-gd}, we have for any $i \in \N, \nabla f_i(\x^*) = 0$. Since $f_i$ is $L$-smooth and by the lemma hypothesis, 
\begin{align}
\label{eq:vt-decomp}
\begin{split}
    \norm{\bv_i^{t,\o}} 
    & = \norm{\nabla f_i(\x_i^{t, \o}) + \p_i^{t, \o} } \\ 
    & \le \norm{\nabla f_i(\x_i^{t, \o}) - \nabla f_i(\x^*)} + \norm{\p_i^{t, \o}} \\
    & \le L \norm{\x_i^{t, \o} - \x^*} + p_{t, \o} \\
    & \le L \left(\norm{\x_i^{t, \o} - \ox^{t, \o}} + \norm{\ox^{t, \o} - \x^*} \right) + p_{t, \o} \\
    & \overset{\eqref{eq:consensus-error-1}}{\le} \gamma_t. 
\end{split}
\end{align}
Thus, $k_i^{t,\o} = 1$ for all $i \in \N$. Then, we take $t_{2,\o}$ as the least $t \ge t_{1,\o}$ such that $\alpha_{t}\le 2/[(1-\rho)(\mu + L)]$. By Lemma \ref{lm:gd-strongcvx}, for $t \ge t_{2, \o}$, from \eqref{eq:avg-update} we have
\begin{align}
\begin{aligned}
\label{eq:local-ctr-upbd}
    & \norm{\ox^{t+1, \o} - \x^*} \\
    \le \ & \norm{\ox^{t, \o} - \frac{\alpha_t}{n} \sum_{i \in \N} \nabla f_i(\ox^{t, \o}) - \x^*} \\
    & + \norm{ \frac{\alpha_t}{n}\sum_{i \in \N}(\nabla f_i(\ox^{t, \o}) - \nabla f_i(\x_i^{t, \o}))} \\
    & + \norm {\frac{\alpha_t}{n} \sum_{i \in \N}\p_i^{t, \o}} + \norm{ \frac{\alpha_t}{n} \sum_{i \in \A} k_i^{t, \o} \bv_i^{t, \o}} \\
    \underset{\eqref{eq:consensus-error-1}}{\le} \ & [1 - \alpha_t \mu (1 - \rho)] \norm{\ox^{t,\o} - \x^*}  \\
    & + c L \sqrt{1 - \rho} \alpha_t^2 \gamma_t +  (1 - \rho)\alpha_t p_{t, \o} + \rho \alpha_t \gamma_t \\
    \underset{\eqref{eq:bargt}}{\le} \ & [1 - \alpha_t \mu( 1 - \rho)] \og_{t, \o} \\
    &  + c L \sqrt{1 - \rho} \alpha_t^2 \gamma_t +  (1 - \rho)\alpha_t p_{t, \o} \\
    & + \rho \alpha_t ( L \og_{t, \o} + p_{t, \o} + c L \sqrt{n} \alpha_t \gamma_t) \\
    = \ & [1 - \alpha_t(\mu (1- \rho) - \rho L)] \og_{t, \o} \\
    & + cL(\sqrt{1 - \rho} + \sqrt{n}\rho) \alpha_t^2 \gamma_t + \alpha_t p_{t, \o}.
\end{aligned}
\end{align}
We next show that, for large enough $t$, \eqref{eq:local-ctr-upbd} implies that $\norm{\ox^{t+1, \o} - \ox^*} \le \og_{t+1, \o}$. Define
\begin{align}
\begin{aligned}
\label{eq:deltat}
    \Delta_{t, \o} & 
    = (t+\varphi)^{\tau_\gamma}\og_{t, \o} \\
    & = \frac{c_\gamma}{L} - \frac{c_{p, \o} }{L(t+\varphi)^{(\tau_\alpha - 2 \tau_\gamma)/3 - \epsilon}} - \frac{c \sqrt{n} c_\alpha c_\gamma}{(t+\varphi)^{\tau_\alpha}}. 
\end{aligned}
\end{align}
Since $\tau_\alpha > 2 \tau_\gamma$ and $\epsilon$ is arbitrarily small, $\Delta_{t, \o}$ is increasing in $t$, and
\begin{align*}
    \og_{t+1, \o} 
    & = \frac{\Delta_{t+1, \o}}{(t+\varphi+1)^{\tau_\gamma}} \\
    & \ge \frac{\Delta_{t, \o}}{(t+\varphi+1)^{\tau_\gamma}} = \Big( \frac{t+\varphi}{t + \varphi+1} \Big)^{\tau_\gamma} \og_{t, \o}.
\end{align*}
By \eqref{eq:local-ctr-upbd}, to show $\norm{\ox^{t+1} - \ox^*} \le \og_{t+1}$, it suffices to have  
\begin{align*}
    & [1 - \alpha_t(\mu (1- \rho) - \rho L)] \og_{t, \o} \\
    & + cL(\sqrt{1 - \rho} + \sqrt{n}\rho) \alpha_t^2 \gamma_t + \alpha_t p_{t, \o} \\
    \le \ &   \Big( \frac{t+\varphi}{t+\varphi+1} \Big)^{\tau_\gamma} \og_{t, \o}.
\end{align*}
By the choice of $t_{2, \o}$, the left hand side of the above inequality is positive. Thus, we can divide $\og_{t, \o} = \Delta_{t, \o}(t+\varphi)^{-\tau_\gamma}$ from both sides and it leads to that
\begin{align}
\label{eq:1-xbound}
\begin{aligned}
    1 - \alpha_t \Big[ \mu(1 - \rho) - \rho L - \frac{cL(\sqrt{1 - \rho} + \sqrt{n} \rho)c_\alpha c_\gamma}{\Delta_{t, \o} (t+\varphi)^{\tau_\alpha}} \\
     - \frac{c_{p, \o}}{\Delta_{t, \o} (t+\varphi)^{(\tau_\alpha - 2 \tau_\gamma)/3 - \epsilon}}\Big] \le \Big( \frac{t+\varphi}{t+\varphi+1} \Big)^{\tau_\gamma}. 
\end{aligned}
\end{align}
By Assumption \ref{as:resilience-ratio} we have $\mu(1 - \rho) - \rho L > 0$, together with $(\tau_\alpha - 2 \tau_\gamma)/3 - \epsilon > 0$, there exists some finite $t_{3, \o} \ge t_{2,\o}$ such that for $t \ge t_{3, \o}$ the left side of \eqref{eq:1-xbound} is strictly less than 1. Using $1 - x \le e^{-x}$ for $x \ge 0$, to show \eqref{eq:1-xbound} it suffices to have
\begin{align}
\label{eq:lnbound}
\begin{aligned}
     \frac{\alpha_t}{\tau_\gamma}\Big[ \mu(1 - \rho) - \rho L - \frac{cL(\sqrt{1 - \rho} + \sqrt{n} \rho)c_\alpha c_\gamma}{\Delta_{t, \o} (t+\varphi)^{\tau_\alpha}} \\ - \frac{c_{p, \o}}{\Delta_{t, \o} (t+\varphi)^{(\tau_\alpha - 2 \tau_\gamma)/3 - \epsilon}}\Big] \ge \ln \frac{t+\varphi+1}{t+\varphi}.
\end{aligned}
\end{align}
Since $\Delta_{t, \o}$ monotonically increases to $c_\gamma/L$, we can find a finite $t_{4, \o}$ as the least $t \ge t_{3,\o}$ such that $\Delta_{t, \o} \ge c_\gamma/(2L)$. Then, for $t \ge t_{4, \o}$, using $\ln(1 + x) \le x$ for $x \ge 0$, to show \eqref{eq:lnbound} it suffices to have
\begin{align*}
    \frac{c_\alpha}{\tau_\gamma}\Big[ \mu(1 - \rho) - \rho L 
    - \frac{2cL^2(\sqrt{1 - \rho} + \sqrt{n} \rho)c_\alpha }{ (t+\varphi)^{\tau_\alpha}}  \\ - \frac{2c_{p, \o}L }{c_\gamma (t+\varphi)^{(\tau_\alpha - 2 \tau_\gamma)/3 - \epsilon}}\Big] \ge \frac{1}{(t+\varphi)^{1 - \tau_\alpha}},
\end{align*}
which holds true for some finite $t_{5, \o} \ge t_{4,\o}$ since $0 < \tau_\alpha < 1$, and thus for all $t \ge t_{5, \o}$. Taking $T_{0, \o} = t_{5,\o}$ concludes the proof. 
\end{proof}

\begin{lemma}
\label{lm:case2}
Choose $\tau_\alpha, \tau_\gamma, \og_t$ as in Lemma \ref{lm:case1}, and in addition $\tau_\alpha + \tau_\gamma < 1$. Suppose that for all $t \ge 0$, we have $\norm{\ox^t - \x^*} > \og_t$. Then, we have for any $0 < \tau < (\tau_\alpha - 2\tau_\gamma)/3$,
\begin{align*}
    \P\left( \lim_{t \rightarrow \infty} (t+1)^{\tau} \norm{\ox^t - \x^*} = 0 \right) = 1.
\end{align*}
\end{lemma}
\begin{proof}
Consider the sample path $\o \in \Omega$ such that Lemma \ref{lm:case1} holds, and since such set of $\o$ has probability measure 1, results holds on such path $\o$ will hold almost surely. Similar to \eqref{eq:vt-decomp}, we have for any $i \in \N$,
\begin{align}
\begin{aligned}
\label{eq:vtbound2}
    \norm{\bv_i^{t, \o}} 
    & \le \norm{\nabla f_i(\x_i^{t, \o})} + p_{t, \o} \\
    & \le L\norm{\ox^{t, \o} - \x^*} + cL \sqrt{n}\alpha_t \gamma_t + p_{t, \o}.
\end{aligned}
\end{align}

\textit{Step I: setup recursion.} Define
\begin{align}
\label{eq:hatkt}
    \hat{k}^{t, \o} = \frac{\gamma_t}{L\norm{\ox^{t, \o} - \x^*} + c L \sqrt{n} \alpha_t\gamma_t + p_{t,\o}}.
\end{align}
Then, by \eqref{eq:vtbound2} we have
\begin{align}
\label{eq:kt-bd1}
    \hat{k}^{t, \o} \le \gamma_t \norm{\bv_i^{t, \o}}^{-1}.
\end{align}
Recall that by the definition of $\og_{t, \o}$ in \eqref{eq:bargt}, 
\begin{align*}
    \gamma_t = L \og_{t, \o} + c L \sqrt{n} \alpha_t \gamma_t + p_{t, \o}. 
\end{align*}
By the lemma hypothesis that  $\norm{\ox^{t, \o} - \x^*} \ge \og_{t, \o}$, we have $\hat{k}^{t, \o} \le 1$. Thus,
\begin{align}
\label{eq:ktbd}
    \hat{k}^{t, \o} \le k_i^{t, \o} =  \min\big(1, \gamma_t\norm{\bv_i^{t, \o}}^{-1}\big).
\end{align}
By Assumption \ref{as:cvx}, for $i \in \N$, $f_i$ is twice differentiable and convex, by the mean value  theorem, there exists some matrix $\M_i^{t, \o} \succeq 0$ such that 
\begin{align*}
    \nabla f_i(\ox^{t, \o}) - \nabla f_i(\x^*) = \M_i^{t, \o} (\ox^{t, \o} - \x^*). 
\end{align*}
Define 
\begin{align*}
    \M^{t, \o} = \frac{1}{|\N|} \sum_{i \in \N} \M_i^{t, \o}. 
\end{align*}
By Assumption \ref{as:cvx} and \ref{as:smooth}, we have
\begin{align*}
    0 \preceq \M_i^{t, \o} \preceq L \I, \ \mu \I \preceq \M^{t, \o} \preceq L\I.
\end{align*}
From \eqref{eq:avg-update}, we have the relation
\begin{align}
\label{eq:main-recursion-case2}
\begin{split}
    & \norm{\ox^{t+1, \o} - \x^*} \\
    \le \ & \norm{\ox^{t, \o} - \frac{\alpha_t}{n} \sum_{i \in \N} k_i^{t, \o} \nabla f_i(\ox^{t, \o}) - \x^*} \\
    &  + \norm{ \frac{\alpha_t}{n}\sum_{i \in \N}k_i^{t, \o}(\nabla f_i(\ox^{t, \o}) - \nabla f_i(\x_i^{t, \o} ))} \\
    &  + \norm {\frac{\alpha_t}{n} \sum_{i \in \N}k_i^{t, \o} \p_i^{t, \o}} + \norm{ \frac{\alpha_t}{n} \sum_{i \in \A} k_i^{t, \o} \bv_i^{t, \o}}.
\end{split}
\end{align}
We have that
\begin{align}
\label{eq:case2cntr1}
\begin{aligned}
& \norm{\ox^{t, \o} - \frac{\alpha_t}{n} \sum_{i \in \N} k_i^{t, \o} \nabla f_i(\ox^{t, \o}) - \x^*} \\ 
\le \ & \norm{\ox^{t, \o} -\x^* - \frac{\alpha_t}{n} \sum_{i \in \N}k_i^{t, \o} \M_i^{t, \o} (\ox^{t, \o} - \x^*)} \\
\le \ & \norm{\I - \frac{\alpha_t}{n} \sum_{i \in \N} k_i^{t, \o} \M_i^{t, \o}}_2 \norm{\ox^{t, \o}  - \x^*  }.
\end{aligned}
\end{align}
Let $\lambda_1(\cdot)$ denote the largest eigenvalue. Take $T_1$ as the least $t$ such that $\alpha_t < 1/[L(1 - \rho)]$, then for $t \ge T_1$, the symmetric matrix $\I - (\alpha_t/n) \sum_{i \in \N} k_i^{t, \o} \M_i^{t, \o}$ is positive definite, and thus
\begin{align*}
    \norm{\I - \frac{\alpha_t}{n} \sum_{i \in \N} k_i^{t, \o} \M_i^{t, \o}}_2 = \lambda_1 \Big( \I - \frac{\alpha_t}{n} \sum_{i \in \N} k_i^{t, \o} \M_i^{t, \o} \Big), \\
    \norm{\I -  \frac{\alpha_t \hat{k}^{t, \o}}{n} \sum_{i \in \N} \M_i^{t, \o}}_2 = \lambda_1\Big( \I -  \frac{\alpha_t \hat{k}^{t, \o}}{n} \sum_{i \in \N} \M_i^{t, \o} \Big).
\end{align*}
Next, we use the fact that for any pair of symmetric matrices $A, B$, 
\begin{align*}
    \lambda_1(A + B) \le \lambda_1(A) + \lambda_1(B).
\end{align*}
It follows that
\begin{align*}
\begin{aligned}
    & \lambda_1\Big(\I - \frac{\alpha_t}{n} \sum_{i \in \N} k_i^{t, \o} \M_i^{t, \o}\Big) \\
    = \ &  \lambda_1\Big( \I -  \frac{\alpha_t \hat{k}^{t, \o}}{n} \sum_{i \in \N} \M_i^{t, \o}\Big) \\
    & + \lambda_1\Big( \frac{\alpha_t}{n} \sum_{i \in \N }(\hat{k}^{t, \o} - k_i^{t, \o}) \M_i^{t, \o} \Big)
\end{aligned}
\end{align*}
Since for all $ i \in \N, \hat{k}^{t, \o} - k_i^{t, \o} \le 0$ and $\M_i^{t, \o}$ is positive semi-definite, we have
\begin{align}
\label{eq:2normbound}
    \norm{\I - \frac{\alpha_t}{n} \sum_{i \in \N} k_i^{t, \o} \M_i^{t, \o}}_2 \le  \norm{\I -  \frac{\alpha_t \hat{k}^{t, \o}}{n} \sum_{i \in \N} \M_i^{t, \o}}_2.
\end{align}
Combing relations \eqref{eq:case2cntr1} and \eqref{eq:2normbound} we obtain
\begin{align*}
    & \norm{\ox^{t, \o} - \frac{\alpha_t}{n} \sum_{i \in \N} k_i^{t, \o} \nabla f_i(\ox^{t, \o}) - \x^*} \\
    \le \ & \norm{\I - \alpha_t (1 - \rho)\hat{k}^{t, \o} \M^{t, \o}}_2 \norm{\ox^{t, \o} - \x^*  } \\
    \le \ & [1 - \alpha_t  \mu (1 - \rho) \hat{k}^{t,\o} ]\norm{\ox^{t, \o}  - \x^*}.
\end{align*} 
Combing the above with
\eqref{eq:ktbd}-\eqref{eq:main-recursion-case2} leads to the recursion
\begin{align}
\begin{aligned}
\label{eq:case2cntr2}
& \norm{\ox^{t+1, \o} - \x^*} \\
\le \ &  [1 - \alpha_t  \mu (1 - \rho) \hat{k}^{t,\o} ]\norm{\ox^{t, \o}  - \x^*} \\
& + c L \sqrt{1 - \rho} \alpha_t^2 \gamma_t +  (1 - \rho)\alpha_t p_{t, \o} + \rho \alpha_t \gamma_t \\
= \ & \big[1 - \alpha_t[\mu(1 - \rho) - \rho L] \hat{k}^{t,\o} \big] \norm{\ox^{t, \o} - \x^*} \\ 
& + cL(\sqrt{1 - \rho} + \rho\sqrt{n}) \alpha_t^2 \gamma_t + \alpha_t p_{t, \o}.
\end{aligned}
\end{align}

\textit{Step II: lower bound for $\hat{k}^{t, \o}$.} We next show that there exists some positive constant $c_{k,\o}$ such that
\begin{align}
\label{eq:hatk-bd}
    \hat{k}^{t, \o} \ge c_{k, \o}(t+\varphi )^{-\tau_\gamma}. 
\end{align}
By the definition of $\hat{k}^{t, \o}$ in \eqref{eq:hatkt}, to show \eqref{eq:hatk-bd} it suffices to show that
\begin{align}
\label{eq:err-bdd}
    \sup _{t \ge 0} \norm{\ox^{t, \o} - \x^*} < \infty, 
\end{align}
which is equivalent to show that $\sup _{t \ge T_1} \norm{\ox^{t, \o} - \x^*} < \infty$. Define the system 
\begin{align*}
    \hat{m}_{t+1, \o} 
    & = \big[1 - \frac{\big(\mu(1 - \rho) - \rho L\big)\alpha_t \gamma_t}{L m_{t, \o} + cL\sqrt{n} \alpha_t \gamma_t + p_{t, \o} } \big] m_{t, \o} \\ 
& \quad+ cL(\sqrt{1 - \rho} + \rho\sqrt{n}) \alpha_t^2 \gamma_t + \alpha_t p_{t, \o},\\
m_{t+1, \o} & = \max\big(\hat{m}_{t+1, \o}, m_{t, \o}\big), 
\end{align*}
with initial condition $m_{T_1, \o} = \norm{\ox^{T_1, \o} - \x^*}$. Note that since $2\tau_\alpha + \tau_\gamma > 4\tau_\alpha/3 + \tau_\gamma/3$, for some constant $c_{m, \o}$ we have
\begin{align*}
   &  cL(\sqrt{1 - \rho} + \rho\sqrt{n}) \alpha_t^2 \gamma_t + \alpha_t p_{t, \o} 
    \\ = \ & c_{m, \o}(t+\varphi)^{-(\frac{4}{3}\tau_\alpha + \frac{1}{3}\tau_\gamma - \epsilon)}. 
\end{align*}
Then, since $\tau_\alpha > 2\tau_\gamma$, the dynamics of $\{m_{t, \o} \}_{t \ge T_1}$ falls into the pursuit of Lemma \ref{lm:bddyanmics} in Appendix and leads to $\sup_{t \ge T_1} m_{t, \o} < \infty$. By \eqref{eq:case2cntr2} and the definition of $m_{t, \o}$, $\forall t \ge T_1, \norm{\ox^{t, \o} - \x^*} \le m_{t, \o}$, and thus \eqref{eq:err-bdd} follows.

\textit{Step III: convergence rate.} Combing \eqref{eq:case2cntr2} and \eqref{eq:hatk-bd}, we have 
\begin{align*}
    & \norm{\ox^{t+1, \o} - \x^*} \\
    \le & \big[1 - c_{k, \o} c_\alpha [\mu(1 - \rho) - \rho L] (t+\varphi)^{-(\tau_\alpha + \tau_\gamma)} \big] \norm{\ox^{t, \o} - \x^*} \\
    & + c_{m, \o}(t+\varphi)^{-(\frac{4}{3}\tau_\alpha + \frac{1}{3}\tau_\gamma - \epsilon)}. 
\end{align*}
Then, with $\tau_\alpha + \tau_\gamma < 1$, we can apply Lemma \ref{lm:cvg_diff_main} in Appendix to obtain the desired convergence rate.
\end{proof} 

\begin{proof}[Proof of Theorem \ref{thm:rdgd}]
From Lemma \ref{lm:case1} and Lemma \ref{lm:case2}, we have for every $0 \le \tau < \min\big(\tau_\gamma, (\tau_\alpha - 2\tau_\gamma)/3 \big)$, 
\begin{align*}
    \P\big( \lim_{t \rightarrow \infty} (t+1)^{\tau} \norm{\ox^t - \x^*} = 0) = 1. 
\end{align*}
By the triangular inequality and \eqref{eq:consensus-error-1}, for every $i \in [n]$
\begin{align*}
    \norm{\x_i^t - \x^*} 
    \le \ &  \norm{\ox^t - \x^*} + \norm{\x_i^t - \ox^t} \\
    \le \ & \norm{\ox^t - \x^*} +  c \sqrt{n}\alpha_t \gamma_t. 
\end{align*}
Since $\tau < \tau_\alpha + \tau_\gamma$, we have
\begin{align*}
    \P\left( \lim_{t \rightarrow \infty} (t+1)^{\tau} \norm{\x_i^t - \x^*} = 0 \right) = 1. 
\end{align*}
\end{proof}

        \section{Conclusion}
\label{sec:conclusion}
In this paper, we have studied a relatively unexplored threat model in distributed optimization with arbitrary gradient attack, and we have proposed a distributed stochastic gradient method \cvrg\ that combines local variance reduced gradient estimation and clipping. We have identified an upper bound of the fraction of attacked agents that relates to the condition number of the aggregated objective function. Under some similarity conditions among local objective functions, we have established the almost sure convergence of \cvrg\ to the exact minimum, which is empirically supported by experiments on both synthetic dataset and image classification datasets. Future directions include extending \cvrg\ to nonconvex or finite sum objective functions where different upper tolerance bounds on the fraction of attacked agents and new similarity conditions may be needed, involving techniques to mitigate the impact of data heterogeneity such as gradient tracking, and understanding the convergence in other notions.

        \bibliographystyle{IEEEtran}
        \bibliography{refs}

        \newpage
        \section*{Appendix}
        \appendix
\begin{lemma}[Lemma 5 in \cite{kar2011convergence}]
\label{lm:cvg_diff_main}
Consider the scalar positive sequences $0 < u_t \le 1$ that
\begin{align*}
    u_t = u_0(t+1)^{-a}, \text{ and } w_t = w_0(t+1)^{-b}, 
\end{align*}
with $0 \le a < 1$ and $a < b$. Then for 
\begin{align*}
    y_{t+1}  = (1 - u_t)y_t + w_t,
\end{align*}
for every $0 < \epsilon < b - a$, $\lim_{t \rightarrow \infty} (t+1)^{b - a - \epsilon} y_t = 0$.  

\end{lemma}

\begin{lemma}[Lemma 25 in \cite{kar2012distributed}]
\label{lm:sum_zero_limit}
Let the sequences $\{u_t\}, \{w_t\}$ be given by 
\begin{align*}
    u_t = u_0(t+1)^{-a}, w_t = w_0(t+1)^{-b}
\end{align*}
where $u_0, w_0, a \ge 0$, and $b > a$. The for arbitrary fixed $j$, 
\begin{align*}
    \lim_{t \rightarrow \infty} \sum_{k=j}^{t-1} \Big[ \Big( \Pi_{l = k+1}^{t - 1} ( 1 - u(t) ) \Big) w_t \Big] = 0.
\end{align*}
\end{lemma}

\begin{lemma}
\label{lm:bddyanmics}
Consider a scalar dynamical system $\{m_t\}$ of the form 
\begin{align*}
    & \hat{m}_{t+1} = \Big(1 -  \frac{u_t}{m_t + v_t} \Big) m_t + w_t, \\
    & m_{t+1} = \max\big( |\hat{m}_{t+1}|, |m_t| \big), 
\end{align*}
where $m_0 > 0$ and $v_t$ is a positive decaying sequence, and 
\begin{align*}
    u_t = \frac{u_0}{(t+1)^{a}}, \ w_t = \frac{w_0}{(t+1)^{b}},
\end{align*}
for some positive constants $b > a, u_0, w_0$. Then, it satisfies that $\sup_{t > 0} m_t < \infty$.
\end{lemma}
\begin{proof}
We prove the lemma by showing that there exists some finite $T$ such that for all $t \ge T$, $m_{t+1} = m_t$. Notice that a sufficient condition for $m_{t+1} = m_t$ is $|\hat{m}_{t+1}| \le |m_t|$ and $m_t > 0$.

By definition, $\{m_t\}$ is a nondecreasing positive sequence, so $m_t + v_t > m_0$. By the definition of $u_t$, take $t_0 = \lceil \big( u_0/m_0 \big)^{1/a} - 1\rceil$, then all $t > t_0$ we have $u_t/(m_t + v_t) < 1$. Then, for $t \ge t_0$, we have $m_t > 0, \hat{m}_t > 0$ and $|\hat{m}_{t+1} \le |m_t|$ reduces to $\hat{m}_{t+1} \le m_t$. 
\begin{align*}
    m_t - \hat{m}_{t+1} = \frac{u_t m_t}{ m_t + v_t} - w_t.
\end{align*}
Since $m_t$ is nondecreasing and $v_t$ is decaying, we have 
\begin{align*}
    \frac{m_t}{m_t + v_t} = \frac{1}{1 + (v_t/m_t)} \ge \frac{m_{t_0}}{m_{t_0} + v_{t_0}} := c_0 > 0. 
\end{align*}
Then, by the definitions of $u_t, w_t$, for
\begin{align*}
    t \ge t_0' := \max\Big(t_0, \lceil \big( \frac{w_0}{c_0 u_0}\big)^{\frac{1}{b - a}} - 1\rceil \Big), 
\end{align*}
we have
\begin{align*}
    m_t - \hat{m}_{t+1} \ge c_0 u_t - w_t > 0.
\end{align*}
Therefore, taking $T = t_0'$ we have for all $t \ge T$, $m_{t+1} = m_t$, and thus $\sup_{t \ge 0} m_t = m_T < \infty$. 
\end{proof}

\end{document}